\documentclass[11pt,a4paper,reqno]{article}
\usepackage{amsmath, amssymb, amsthm}
\usepackage{fullpage}
\usepackage{hyperref}
\usepackage{pst-node}
\usepackage{color}

\usepackage[all]{xy}
\usepackage{changes} 
\colorlet{Changes@Color}{red}
\usepackage{tikz-cd}
\tikzcdset{arrow style=tikz, diagrams={>=stealth}}

\newcommand{\beq}{\begin{equation}}
\newcommand{\eeq}{\end{equation}}

\newcommand{\ot}{{\otimes}}

\numberwithin{equation}{section}

\theoremstyle{plain}
\newtheorem{thm}{Theorem}[section]
\newtheorem{lem}[thm]{Lemma}

 \newtheorem{cor}[thm]{Corollary}
\newtheorem{defi}[thm]{Definition}

\theoremstyle{remark}
\newtheorem{exa}[thm]{Example}
\newtheorem{rem}[thm]{Remark}

\newcommand{\di}{{\diamond_{B}}}
\newcommand{\BB}{{\overline{B}}}
\newcommand{\lan}{{\langle}}
\newcommand{\ran}{{\rangle}}

\DeclareMathOperator{\Hom}{Hom}
\newcommand*{\id}{\textup{id}}

\newcommand{\one}[1]{{#1}{}_{\scriptscriptstyle{(1)}}}
\newcommand{\two}[1]{{#1}{}_{\scriptscriptstyle{(2)}}}
\newcommand{\three}[1]{{#1}{}_{\scriptscriptstyle{(3)}}}

\renewcommand{\o}{{}_{\scriptscriptstyle{(1)}}}
\renewcommand{\t}{{}_{\scriptscriptstyle{(2)}}}
\renewcommand{\th}{{}_{\scriptscriptstyle{(3)}}}

\newcommand{\la}{{\triangleright}}
\newcommand{\ra}{{\triangleleft}}
\newcommand{\eps}{{\varepsilon}}
\DeclareMathOperator{\tens}{\otimes}

\newcommand{\M}{\mathcal{M}}

\newcommand{\cL}{\mathcal{L}}
\newcommand{\CL}{\mathcal{L}}

\newcommand{\CM}{\mathcal{M}}

\renewcommand{\>}{\rangle}
\newcommand{\rlbicross}{\triangleright\!\!\!\blacktriangleleft}
\allowdisplaybreaks

\title{{\bf Coquasi-bialgebroids and cocycle twisting}}

\author{Xiao Han${}^*$ and Shahn Majid\footnote{x.han@qmul.ac.uk, s.majid@qmul.ac.uk}
\\ \  \\
School of Mathematical Sciences, Queen Mary University of London\\ London E1 4NS, UK}

\begin{document}

\date{}

\maketitle
\parskip = .75 ex


\begin{abstract}
We introduce coquasi-bialgebroids over a noncommutative base algebra. 
Using Takeuchi's \(\times_B\)-coalgebra formalism, we require the coproduct to remain an algebra map into the Takeuchi product, while the product is associative only up to an invertible normalized \(3\)-cocycle. This gives a bialgebroid analogue of coquasi-bialgebras and provides a natural framework for cocycle-twisted bialgebroid constructions.

We develop the basic theory and prove a twisting theorem by convolution-invertible \(2\)-cochains. As a main class of examples, we construct coquasi Connes--Moscovici-type bialgebroids on \(B\otimes H\otimes B\), where \(H\) is a coquasi-bialgebra measuring an algebra \(B\), with twisting data \(\gamma:H\otimes H\to B\). We also give finite-group examples arising from a subgroup \(G\subseteq X\) and a choice of transversal. Finally, under finite projectivity assumptions, we describe the dual quasi-bialgebroid construction and its relation to Drinfeld-type twisting.

\end{abstract}

\begin{quote}  2010 Mathematics Subject Classification:  16T05, 16T15, 18D10.\end{quote}

\section{Introduction}

Hopf algebroids and bialgebroids provide a natural generalisation of Hopf algebras in which the base field is replaced by a possibly noncommutative algebra. This enlargement is essential in many geometric and categorical examples: groupoids rather than groups, Lie--Rinehart algebras rather than Lie algebras, quantum principal bundles rather than Galois objects. In such situations the coproduct no longer takes values in an ordinary tensor product over the ground field, but in a balanced tensor product over the base algebra, and the correct algebraic target of the coproduct is the Takeuchi product.  This leads to the \(\times_B\)-bialgebra, or bialgebroid, formalism introduced by Takeuchi \cite{Ta} and developed further in \cite{BM, Lu,schau4, Xu}.

A parallel development in Hopf algebra theory is the passage from Hopf algebras and bialgebras to quasi-Hopf, coquasi-Hopf, quasi-bialgebra and coquasi-bialgebra structures. In a coquasi-bialgebra, the coproduct remains coassociative, but the algebra product is associative only up to a convolution-invertible \(3\)-cocycle. Such structures arise naturally from group-theoretical constructions, monoidal categories, bicrossproduct constructions and gauge transformations. They also provide a convenient language for describing twists and associators in situations where strict associativity is too rigid.

The aim of this paper is to combine these two directions. We introduce a notion of coquasi-bialgebroid over a noncommutative base algebra \(B\). Roughly speaking, a left \(B\)-coquasi-bialgebroid is a \(\times_B\)-coalgebra \(\cL\) equipped with source and target maps from \(B\) and \(B^{\mathrm{op}}\), respectively. The failure of associativity is controlled by an invertible normalized \(3\)-cocycle $\Phi\in \operatorname{Hom}_{B^e-}(\cL\otimes_{B^e}\cL\otimes_{B^e}\cL,B)$.
The ordinary left bialgebroid case is recovered when \(\Phi\) is trivial. This definition is designed so that the coproduct remains compatible with the product, while the associativity constraint is weakened in exactly the way familiar from coquasi-bialgebras.

One subtle point is that, over a noncommutative base, the tensor product formalism is more delicate than over a field. The coproduct must land in a Takeuchi product, and iterated coproducts require care because the product \(\times_B\) is not strictly associative in general. We therefore begin by fixing the necessary algebraic conventions for \(B^e\)-modules, balanced tensor products and Takeuchi products. This makes the definition of a \(B\)-coalgebroid precise and allows the \(3\)-cocycle conditions for a coquasi-bialgebroid to be written in a form compatible with the base algebra.

The first main result is a twisting theorem for coquasi-bialgebroids. Given a left \(B\)-coquasi-bialgebroid \((\cL,\Phi)\) and a suitable normalized convolution-invertible \(2\)-cochain $\Sigma\in \operatorname{Hom}_{B^e-}(\cL\otimes_{B^e}\cL,B)$,
we define a new product on the underlying \(B\)-coalgebroid \(\cL\). The resulting object is again a left \(B\)-coquasi-bialgebroid, with a new \(3\)-cocycle \(\Phi^\Sigma\). In particular, starting from an ordinary bialgebroid, a \(2\)-cochain produces a canonical \(3\)-cocycle \(\partial\Sigma\). This is the bialgebroid analogue of the familiar passage from \(2\)-cochains to \(3\)-cocycles in the theory of coquasi-bialgebras.

The second main part of the paper concerns constructions. We show that coquasi-bialgebra data give rise to coquasi versions of Connes--Moscovici-type bialgebroids. Let \((H,\phi)\) be a coquasi-bialgebra and let \(B\) be an associative algebra equipped with a measuring action of \(H\). Suppose moreover that there is a convolution-invertible map $\gamma:H\otimes H\to B$
controlling the failure of the action to be strict. We construct a left \(B\)-coquasi-bialgebroid  structure on $B\otimes H\otimes B$.
The source, target, coproduct and counit are the expected Connes--Moscovici-type ones, while the product is twisted by \(\gamma\). The associator \(\Phi^\gamma\) is then computed explicitly in terms of the associator \(\phi\) of \(H\), the measuring of \(H\) on \(B\), and the twisting data \(\gamma\). When \(\gamma\) is trivial, this gives a particularly transparent construction associated to a strict module algebra over a coquasi-bialgebra.

These constructions are not merely formal. As a concrete family of examples, we use the group-theoretical coquasi-Hopf algebras associated to a finite group \(X\), a subgroup \(G\subseteq X\), and a choice of transversal \(M\) such that \(X=GM\). The associated bicrossproduct coquasi-Hopf algebra acts by measuring on the group algebra \(kG\), and the above construction yields explicit coquasi Connes--Moscovici-type bialgebroids over \(kG\). These examples illustrate how nontrivial associators arise naturally from factorisation data in finite group theory.

We also discuss the relation with ordinary Hopf-algebroid cocycle twisting. Under additional
compatibility assumptions, the twisted coquasi-bialgebroid structure on \(B \otimes H \otimes B\)
can be viewed as arising from a \(2\)-cochain twist of an ordinary Connes--Moscovici-type
bialgebroid. In special cases, when the resulting associator becomes trivial, one recovers
ordinary left and anti-left Hopf algebroids. This connects the present coquasi construction
with previously studied cocycle-twisted Hopf algebroids and cleft extensions of Hopf algebroids
\cite{BB,HM22}.

Finally, we give the dual construction. Under finite projectivity assumptions, the dual of a coquasi-bialgebroid is a quasi-bialgebroid. On the dual side, the algebra remains associative, 
while coassociativity of the coproduct is controlled by an invertible associator $\Phi\in \cL\times_B \cL\times_B \cL$. We show that cocycle twisting on the coquasi side corresponds to Drinfeld-type twisting on the quasi side. Thus the paper gives both the coquasi and quasi versions of the theory, extending the familiar duality between coquasi-bialgebras and quasi-bialgebras to the bialgebroid setting.

The paper is organised as follows. Section~2 recalls the necessary preliminaries on \(B^e\)-modules, Takeuchi products and \(\times_B\)-coalgebras, then introduces left \(B\)-coquasi-bialgebroids and proves the general twisting theorem. Section~3 constructs coquasi Connes--Moscovici-type bialgebroids from coquasi-bialgebras acting by measuring on an algebra \(B\), and gives finite group examples arising from subgroup-transversal data. Section~4 gives the dual theory of quasi-bialgebroids and explains how finite duality exchanges the coquasi cocycle twisting with quasi-bialgebroid Drinfeld twisting.

\section{Coquasi-bialgebroids}
The purpose of this section is to set up the general algebraic framework for coquasi-bialgebroids over a noncommutative base algebra. Since the base is not assumed to be commutative, 
the usual tensor-product conventions for coalgebras over a field have to be replaced by balanced tensor products and Takeuchi products. We first recall the necessary notation for
 \(B^e\)-modules and \(\times_B\)-coalgebras, and then introduce left \(B\)-coquasi-bialgebroids as bialgebroid-like objects whose associativity is controlled by a normalized
 convolution-invertible \(3\)-cocycle.

\subsection{Algebraic preliminaries} \label{sec2}

We begin by recalling the basic definitions and notation used throughout the paper. Let $B$ be an unital algebra over a field $k$. We denote the opposite algebra by $\BB$ and let $B\to \BB$, $b\mapsto\Bar{b}$ for any $b\in B$ be the obvious $k$-algebra antiisomorphism. Define $B^{e}:=B\ot \BB$, so $B$ and $\BB$ are obvious subalgebras of $B^{e}$. Let $M, N$ be $B^{e}$-bimodules. We define
\begin{align*}
    M\di N:=\int_{b} {}_{\Bar{b}}M\ot {}_{b}N:=&M\ot N/\langle \Bar{b}m\ot n-m\ot bn|b\in B, m\in M, n\in N\rangle\\
    M\ot_{B} N:=\int_{b} M_{b}\ot {}_{b}N:=&M\ot N/\langle mb\ot n-m\ot bn|b\in B, m\in M, n\in N\rangle\\
    M\ot_{\BB} N:=\int_{b} M_{\Bar{b}}\ot {}_{\Bar{b}}N:=&M\ot N/\langle m\Bar{b}\ot n-m\ot \Bar{b}n|b\in B, m\in M, n\in N\rangle\\
\end{align*}
For convenience, we also define $N\ot^{B}M=\int_{b} {}_{b}N\ot M_{b}$ and  $N\ot^{\BB}M=\int_{b} {}_{\Bar{b}}N\ot M_{\Bar{b}}$. Moreover, we define
\begin{align*}
    \int^{b}M_{\Bar{b}}\ot N_{b}:=\{\sum_{i}m_{i}\ot n_{i}\in M\ot N\quad |\quad m_{i}\Bar{b}\ot n_{i}=m_{i}\ot n_{i}b, \forall b\in B\}.
\end{align*}

The symbol $\int^{b}$ and $\int^{c}$ commute, also $\int_{b}$ and $\int_{c}$ commute. However, in general, the symbol $\int^{b}$ and $\int_{c}$ doesn't commute. For any $\BB$-bimodule $M$ and any $B$-bimodule $N$, we also define
\begin{align*}
    M\times_{B}N:=\int^{a}\int_{b} {_{\Bar{b}}M_{\Bar{a}}}\ot {_bN_a}.
\end{align*}
$M\times_{B}N$ is called Takeuchi product of $M$ and $N$. If $P$ is a $B^{e}$-bimodule, then $P\times_{B}N$ is a $B$-bimodule with $B$ acting on $P$. Similarly, $M\times_{B}P$ is a $\BB$-bimodule with $\BB$ acting on $P$. If both $M$ and $N$ are $B^{e}$-bimodule, then $M\times N$ is also a $B^{e}$-bimodule. However, the product \(\times_B\) is neither associative nor unital on the category of \(B^e\)-bimodules. For any $M,N, P\in {}_{B^{e}}\M_{B^{e}}$, we can define
\begin{align*}
    M\times_{B}P\times_{B}N:=\int^{a,b}\int_{c,d} {}_{\Bar{c}}M_{\Bar{a}}\ot {}_{c,\Bar{d}}P_{a, \Bar{b}}\ot {}_{d}N_{b},
\end{align*}
where $\int^{a,b}:=\int^{a}\int^{b}$ and $\int_{c,d}:=\int_{c}\int_{d}$. There are maps
\begin{align*}
    &\alpha:(M\times_{B}P)\times_{B} N\to M\times_{B}P\times_{B}N,\quad m\ot p\ot n\mapsto m\ot p\ot n,\\
    &\alpha':M\times_{B}(P\times N)\to M\times_{B}P\times_{B}N,\quad m\ot p\ot n\mapsto m\ot p\ot n.\\
\end{align*}
Notice that neither $\alpha$ nor $\alpha'$ are isomorphisms  in general. In particular, if all $B$-module and $\BB$-module structures are faithfully flat, then $\alpha$ and $\alpha'$ are isomorphisms. 

\subsection{Coalgebroids}

Here we recall the basic definitions (cf. \cite{BW}, \cite{Boehm}). Let $B$ be a unital algebra over a field $k$.
A {\em $B$-ring}   means a unital algebra in the monoidal category ${}_B\CM_B$ of $B$-bimodules. Likewise,  a  {\em $B$-coring} is a coalgebra in ${}_B\CM_B$. Morphisms of $B$-(co)rings are defined to be morphisms of (co)algebras, but in the category ${}_B\CM_B$.

To specify a unital \(B\)-ring \(\cL\) is equivalently to specify a unital \(k\)-algebra \(\cL\) together with an algebra map \(\eta:B\to \cL\). Left and right multiplication in $\cL$ pull back to left and right $B$-actions as a bimodule (so $bXc=\eta(b)X\eta(c)$ for all $b,c\in B$ and $X\in \cL$) and the product descends to the product $\mu_B:\cL\tens_B\cL\to \cL$ with $\eta$ the unit map. Conversely, given
$\mu_B$ we can pull back to an associative product on $\cL$ with unit $\eta(1)$.

Now suppose that $s:B\to \cL$ and $t:\BB\to \cL$ are algebra maps with images that commute. Then $\eta(b\tens c)=s(b)t(c)$ is an algebra map $\eta: B^e\to \cL$, where $B^e=B\tens \BB$, and is equivalent to making $\cL$ a $B^e$-ring. The left $B^e$-action part of this is equivalent to a $B$-bimodule structure
\begin{equation}\label{eq:rbgd.bimod}
b.X.c=b\Bar{c}X:= s(b) t(c)X
\end{equation}
for all $b,c\in B$ and $X\in \cL$. 


If $\cL$ and $\mathcal R$ are two $B^e$-rings, then the Takeuchi product $\cL\times_B\mathcal R$ is an algebra with multiplication given by $(\ell\otimes r)(\ell'\otimes r')=\ell\ell' \otimes rr'$ for $\ell\ot r,\ell'\ot r'\in \cL\times_B\mathcal R$.

\begin{defi}\cite{Ta}\label{def:left.bgd} Let $B$ be a unital algebra. A left \(B\)-coalgebroid, or \(\times_B\)-coalgebra, is a \(B^e\)-bimodule \(\cL\) equipped with two left \(B^e\)-module maps, a coproduct \(\Delta:\cL\to \cL\times_B \cL\) and a counit
 \(\varepsilon:\cL\to B\), satisfying
\begin{align*}
(1)\quad&\alpha\circ(\Delta\times_{B}\id)\circ\Delta=\alpha'\circ(\id\times_{B}\Delta)\circ\Delta:\cL\to \cL\times_{B}\cL\times_{B}\cL\\    (2)\quad&\varepsilon(X\o)X\t=\overline{\varepsilon(X\t)}X\o=X,
\end{align*}
for any $X\in\cL$, where we use the sumless Sweedler index to denote the coproduct, namely, $\Delta(X)=X\o\ot X\t$.

\end{defi}

In particular, a $B$-coalgebroid is a $B$-coring with the bimodule structure given by (\ref{eq:rbgd.bimod}).  Given a  $B$-coalgebroid $\cL$,   there is an algebra structure on $\Hom_{B^{e}-}(\cL\otimes_{B^e}\cL, B)$ with the (convolution) product and unit given by
\begin{align}
    f\star g(X, Y):=f(\one{X}, \one{Y})g(\two{X}, \two{Y}),\quad \tilde{\varepsilon}(X, Y)= \varepsilon(XY)
\end{align}
for all $X, Y\in \cL$ and $f, g\in \Hom_{B^{e}-}(\cL\otimes_{B^e}\cL, B)$.

\subsection{\bf Coquasi bialgebroids}\label{seccoquasi}

We now introduce the main general definition of the paper. A coquasi-bialgebroid is obtained from a left bialgebroid by weakening strict associativity of the algebra structure.
 The failure of associativity is controlled by a normalized convolution-invertible \(3\)-cocycle with values in the base algebra \(B\). When this \(3\)-cocycle is trivial, the definition 
reduces to the usual notion of a left bialgebroid. We also prove that convolution-invertible \(2\)-cochains twist coquasi-bialgebroids to new coquasi-bialgebroids, in direct analogy
 with the classical theory of coquasi-bialgebras.

\begin{defi}\label{def:coright.bgd} Let $B$ be a unital associative algebra. A left $B$-coquasi-bialgebroid $\cL$ is a unital not-necessarily associative algebra  equipped with algebra maps  $s:B\to \cL$ and $t:\overline{B}\to \cL$ such that $s(b)t(b')=t(b')s(b),\forall b,b'\in B$. Moreover, $\cL$ is a $\times_{B}$-coalgebra with the bimodule structure (\ref{eq:rbgd.bimod}) which is compatible in the sense:
\begin{itemize}
\item[(i)] The coproduct $\Delta: \cL\to \cL\times_B \cL$
 is an algebra map.
\item[(ii)] The counit $\varepsilon$ satisfies
\begin{equation*}\varepsilon(1_{\cL})=1_{B},\quad  \varepsilon(X\,\varepsilon(Y))=\varepsilon(XY)=\varepsilon(X\,\overline{\varepsilon(Y)})\end{equation*}
for all $X,Y\in \cL$ and $b\in B$.
\item[(iii)] There is an `\textup{invertible normalised 3-cocycle}' on
$\cL$ in the sense of a convolution invertible element $\Phi\in \Hom_{B^{e}-}(\cL\otimes_{B^e}\cL\otimes_{B^e}\cL, B)$ such that
\begin{align*}(1)\quad &\Phi(1_{\cL}, X, Y)=\Phi(X, 1_{\cL}, Y)=\Phi(X, Y, 1_{\cL})=\varepsilon(XY),\\
(2)\quad &\Phi(X, Y, Z\, b)=\Phi(X, Y, Z\, \overline{b}),\\
(3)\quad& \Phi(X_{\scriptscriptstyle{(1)}}, Y_{\scriptscriptstyle{(1)}}, Z_{\scriptscriptstyle{(1)}}W_{\scriptscriptstyle{(1)}})\Phi(X_{\scriptscriptstyle{(2)}}Y_{\scriptscriptstyle{(2)}}, Z_{\scriptscriptstyle{(2)}}, W_{\scriptscriptstyle{(2)}})\\
&=\Phi(X_{\scriptscriptstyle{(1)}}, \Phi(Y_{\scriptscriptstyle{(1)}}, Z_{\scriptscriptstyle{(1)}}, W_{\scriptscriptstyle{(1)}})Y_{\scriptscriptstyle{(2)}}Z_{\scriptscriptstyle{(2)}}, W_{\scriptscriptstyle{(2)}})\Phi(X_{\scriptscriptstyle{(2)}}, Y_{\scriptscriptstyle{(3)}}, Z_{\scriptscriptstyle{(3)}}),\\
(4)\quad & \Phi(X_{\scriptscriptstyle{(1)}}, Y_{\scriptscriptstyle{(1)}}, Z_{\scriptscriptstyle{(1)}})\,(X_{\scriptscriptstyle{(2)}}Y_{\scriptscriptstyle{(2)}})Z_{\scriptscriptstyle{(2)}}=\overline{\Phi(X_{\scriptscriptstyle{(2)}}, Y_{\scriptscriptstyle{(2)}}, Z_{\scriptscriptstyle{(2)}})}\,X_{\scriptscriptstyle{(1)}}(Y_{\scriptscriptstyle{(1)}}Z_{\scriptscriptstyle{(1)}})
\end{align*}
for all $X, Y, Z, W\in \cL$.
\end{itemize}
If $\Phi$ is trivial, i.e. $\Phi(X, Y, Z)=\varepsilon(XYZ)$, then $\cL$ is a left bialgebroid. A left bialgebroid $\cL$ over $B$ is a left Hopf algebroid (\cite{schau1}, Thm and Def 3.5.) if
\[\lambda: \cL\ot_{\BB}\cL\to \cL\di\cL,\quad
    \lambda(X\ot Y)=\one{X}\ot \two{X}Y\]
is invertible. Similarly, A left bialgebroid $\cL$ over $B$ is an anti-left Hopf algebroid if
\[\mu: \cL\ot_{B}\cL\to \cL\di \cL,\quad
    \mu(X\ot Y)=\one{X}Y\ot \two{X}\]
is invertible. 
\end{defi}

\begin{rem}   
Here (2) implies the minimal condition in that $\Phi^{-1}$ should be a right-handed 3-cocycle in the sense
\begin{align*}
    (X_{\scriptscriptstyle{(1)}}&Y_{\scriptscriptstyle{(1)}}, Z_{\scriptscriptstyle{(1)}}, W_{\scriptscriptstyle{(1)}})\Phi^{-1}(X_{\scriptscriptstyle{(2)}}, Y_{\scriptscriptstyle{(2)}}, Z_{\scriptscriptstyle{(2)}}W_{\scriptscriptstyle{(2)}})\\
&=\Phi^{-1}(X_{\scriptscriptstyle{(1)}}, Y_{\scriptscriptstyle{(1)}}, Z_{\scriptscriptstyle{(1)}})\Phi^{-1}(X_{\scriptscriptstyle{(2)}},
 \overline{\Phi^{-1}(Y_{\scriptscriptstyle{(3)}}, Z_{\scriptscriptstyle{(3)}}, W_{\scriptscriptstyle{(2)}})}\,Y_{\scriptscriptstyle{(2)}}Z_{\scriptscriptstyle{(2)}}, W_{\scriptscriptstyle{(1)}}).
\end{align*}
We can also see $\varepsilon((XY)Z)=\varepsilon(X(YZ))=:\varepsilon(XYZ)$ (and they both equal to $\varepsilon(X \,\overline{\varepsilon(Y \,\overline{\varepsilon(Z)}\,)}\,)$). Here (1) and (3) imply that
\[\Phi(X, Y, 1_{\cL})=\varepsilon(XY).\]
\end{rem}
\begin{thm}\label{twistprod1}
    Let $(\cL, \Phi)$ be a left $B$-coquasi-bialgebroid, and $\Sigma$ be an  unital convolution invertible element in $\Hom_{B^{e}-}(\cL\otimes_{B^e}\cL, B)$, such that $\Sigma(X, Y\,\overline{b})=\Sigma(X, Y\,b)$. If we define a new product on $\cL$ by
    \begin{align*}
    X\cdot_{\Sigma} Y:=\Sigma(\one{X}, \one{Y})\overline{\Sigma^{-1}(\three{X}, \three{Y})}\two{X}\two{Y},
\end{align*}
then the underlying \(B\)-coalgebroid \(\cL\), equipped with the new product, is a left \(B\)-coquasi-bialgebroid, denoted by \(\cL^\Sigma\), with the twisted 3-cocycle given by
\begin{align*}
   \Phi^{\Sigma}(X, Y, Z)=\Sigma(X_{\scriptscriptstyle{(1)}}, \Sigma(Y_{\scriptscriptstyle{(1)}}, Z_{\scriptscriptstyle{(1)}})Y_{\scriptscriptstyle{(2)}}Z_{\scriptscriptstyle{(2)}})\Phi(X_{\scriptscriptstyle{(2)}}, Y_{\scriptscriptstyle{(3)}}, Z_{\scriptscriptstyle{(3)}})\Sigma^{-1}(\overline{\Sigma^{-1}(X_{\scriptscriptstyle{(4)}}, Y_{\scriptscriptstyle{(5)}}})\,X_{\scriptscriptstyle{(3)}}Y_{\scriptscriptstyle{(4)}}, Z_{\scriptscriptstyle{(4)}}),
\end{align*}
for all $X, Y, Z\in \cL$.
\end{thm}

\begin{proof}
A direct computation shows that the coproduct is compatible with the new product and that its image still lies in $\cL\times_{B}\cL$.  We first check $\Sigma^{-1}(X, Y\,\overline{b})=\Sigma^{-1}(X, Y\, b)$. Indeed,
\begin{align*}
    \Sigma^{-1}(X, Y\,\overline{b})=&\Sigma^{-1}(X\o, Y\o\,\overline{b})\Sigma(X\t, Y\t)\Sigma^{-1}(X\th, Y\th)\\
    =&\Sigma^{-1}(X\o, Y\o)\Sigma(X\t, Y\t\,b)\Sigma^{-1}(X\th, Y\th)\\
    =&\Sigma^{-1}(X\o, Y\o)\Sigma(X\t, Y\t\,\overline{b})\Sigma^{-1}(X\th, Y\th)\\
    =&\Sigma^{-1}(X\o, Y\o)\Sigma(X\t, Y\t\,)\Sigma^{-1}(X\th, Y\th\, b)\\
    =&\Sigma^{-1}(X, Y\,b).
\end{align*}
It follows directly that \(\Phi^\Sigma\) is unital and belongs to $\Hom_{B^{e}-}(\cL\otimes_{B^e}\cL\otimes_{B^e}\cL, B)$, moreover it satisfies (2). We can show the inverse of $\Phi^{\Sigma}$ is given by
\begin{align}
    (\Phi^{\Sigma})^{-1}(X, Y, Z)=\Sigma(\Sigma(X_{\scriptscriptstyle{(1)}}, Y_{\scriptscriptstyle{(1)}})X_{\scriptscriptstyle{(2)}}Y_{\scriptscriptstyle{(2)}}, Z_{\scriptscriptstyle{(1)}})\Phi^{-1}(X_{\scriptscriptstyle{(3)}}, Y_{\scriptscriptstyle{(3)}}, Z_{\scriptscriptstyle{(2)}})\Sigma^{-1}(X_{\scriptscriptstyle{(4)}}, \overline{\Sigma^{-1}(Y_{\scriptscriptstyle{(5)}}, Z_{\scriptscriptstyle{(4)}}})\,Y_{\scriptscriptstyle{(4)}}Z_{\scriptscriptstyle{(3)}}),
\end{align}
where we use the fact that
\begin{align*}
    \Sigma(X_{\scriptscriptstyle{(1)}}, \Sigma(Y_{\scriptscriptstyle{(1)}}, Z_{\scriptscriptstyle{(1)}})Y_{\scriptscriptstyle{(2)}}Z_{\scriptscriptstyle{(2)}})\Sigma^{-1}(X_{\scriptscriptstyle{(2)}}, \overline{\Sigma^{-1}(Y_{\scriptscriptstyle{(4)}}, Z_{\scriptscriptstyle{(4)}})}Y_{\scriptscriptstyle{(3)}}Z_{\scriptscriptstyle{(3)}})&=\varepsilon(XYZ),\\
    \Sigma^{-1}(X_{\scriptscriptstyle{(1)}}, \overline{\Sigma^{-1}(Y_{\scriptscriptstyle{(2)}}, Z_{\scriptscriptstyle{(2)}})}Y_{\scriptscriptstyle{(1)}}Z_{\scriptscriptstyle{(1)}})\Sigma(X_{\scriptscriptstyle{(2)}}, \Sigma(Y_{\scriptscriptstyle{(3)}}, Z_{\scriptscriptstyle{(3)}})Y_{\scriptscriptstyle{(4)}}Z_{\scriptscriptstyle{(4)}})&=\varepsilon(XYZ).
\end{align*}
We next show (4) for $\Phi^\Sigma$:
\begin{align*}
    \overline{\Phi^{\Sigma}(X_{\scriptscriptstyle{(2)}}, Y_{\scriptscriptstyle{(2)}}, Z_{\scriptscriptstyle{(2)}})}&X_{\scriptscriptstyle{(1)}}\cdot_\Sigma(Y_{\scriptscriptstyle{(1)}}\cdot_\Sigma Z_{\scriptscriptstyle{(1)}})\\
    = & \overline{\Sigma(X_{\scriptscriptstyle{(4)}}, s(\Sigma(Y_{\scriptscriptstyle{(6)}}, Z_{\scriptscriptstyle{(6)}}))Y_{\scriptscriptstyle{(7)}}Z_{\scriptscriptstyle{(7)}})\Phi(X_{\scriptscriptstyle{(5)}}, Y_{\scriptscriptstyle{(8)}}, Z_{\scriptscriptstyle{(8)}})\Sigma^{-1}(t(\Sigma^{-1}(X_{\scriptscriptstyle{(7)}}, Y_{\scriptscriptstyle{(10)}}))X_{\scriptscriptstyle{(6)}}Y_{\scriptscriptstyle{(9)}}, Z_{\scriptscriptstyle{(9)}})}\\
    &\quad \Sigma(X_{\scriptscriptstyle{(1)}}, \Sigma(Y_{\scriptscriptstyle{(1)}}, Z_{\scriptscriptstyle{(1)}})\,Y_{\scriptscriptstyle{(2)}}Z_{\scriptscriptstyle{(2)}}) \overline{\Sigma^{-1}(X_{\scriptscriptstyle{(3)}},  \overline{\Sigma^{-1}(Y_{\scriptscriptstyle{(5)}}, Z_{\scriptscriptstyle{(5)}})}Y_{\scriptscriptstyle{(4)}}Z_{\scriptscriptstyle{(4)}})}\,X_{\scriptscriptstyle{(2)}}(Y_{\scriptscriptstyle{(3)}}Z_{\scriptscriptstyle{(3)}})\\
    =& \overline{\Phi(X_{\scriptscriptstyle{(3)}}, Y_{\scriptscriptstyle{(4)}}, Z_{\scriptscriptstyle{(4)}})\Sigma^{-1}( \overline{\Sigma^{-1}(X_{\scriptscriptstyle{(5)}}, Y_{\scriptscriptstyle{(6)}})}\,X_{\scriptscriptstyle{(4)}}Y_{\scriptscriptstyle{(5)}}, Z_{\scriptscriptstyle{(5)}})}\\
    &\quad
    \Sigma(X_{\scriptscriptstyle{(1)}}, \Sigma(Y_{\scriptscriptstyle{(1)}}, Z_{\scriptscriptstyle{(1)}})Y_{\scriptscriptstyle{(2)}}Z_{\scriptscriptstyle{(2)}})X_{\scriptscriptstyle{(2)}}(Y_{\scriptscriptstyle{(3)}}Z_{\scriptscriptstyle{(3)}})\\
    =& \overline{\Sigma^{-1}( \overline{\Sigma^{-1}(X_{\scriptscriptstyle{(5)}}, Y_{\scriptscriptstyle{(6)}})}\,X_{\scriptscriptstyle{(4)}}Y_{\scriptscriptstyle{(5)}}, Z_{\scriptscriptstyle{(5)}})}
    \Sigma(X_{\scriptscriptstyle{(1)}}, \Sigma(Y_{\scriptscriptstyle{(1)}}, Z_{\scriptscriptstyle{(1)}})Y_{\scriptscriptstyle{(2)}}Z_{\scriptscriptstyle{(2)}})\\
    &\quad  \overline{\Phi(X_{\scriptscriptstyle{(3)}}, Y_{\scriptscriptstyle{(4)}}, Z_{\scriptscriptstyle{(4)}})}\,X_{\scriptscriptstyle{(2)}}(Y_{\scriptscriptstyle{(3)}}Z_{\scriptscriptstyle{(3)}})\\
    =& \overline{\Sigma^{-1}( \overline{\Sigma^{-1}(X_{\scriptscriptstyle{(5)}}, Y_{\scriptscriptstyle{(6)}})}\,X_{\scriptscriptstyle{(4)}}Y_{\scriptscriptstyle{(5)}}, Z_{\scriptscriptstyle{(5)}})}\\
    &\quad
    \Sigma(X_{\scriptscriptstyle{(1)}}, \Sigma(Y_{\scriptscriptstyle{(1)}}, Z_{\scriptscriptstyle{(1)}})Y_{\scriptscriptstyle{(2)}}Z_{\scriptscriptstyle{(2)}})\,\Phi(X_{\scriptscriptstyle{(2)}}, Y_{\scriptscriptstyle{(3)}}, Z_{\scriptscriptstyle{(3)}})(X_{\scriptscriptstyle{(3)}}Y_{\scriptscriptstyle{(4)}})Z_{\scriptscriptstyle{(4)}}\\
    =&\Phi^{\Sigma}(X_{\scriptscriptstyle{(1)}}, Y_{\scriptscriptstyle{(1)}}, Z_{\scriptscriptstyle{(1)}})(X_{\scriptscriptstyle{(2)}}\cdot_{\Sigma}Y_{\scriptscriptstyle{(2)}})\cdot_{\Sigma}Z_{\scriptscriptstyle{(2)}}.
\end{align*}
Finally, we show (3) for $\Phi^\Sigma$. On the one hand,
\begin{align*}
    \Phi^{\Sigma}(X_{\scriptscriptstyle{(1)}}&, Y_{\scriptscriptstyle{(1)}}, Z_{\scriptscriptstyle{(1)}}\cdot_\Sigma W_{\scriptscriptstyle{(1)}})\Phi^{\Sigma}(X_{\scriptscriptstyle{(2)}}\cdot_\Sigma Y_{\scriptscriptstyle{(2)}}, Z_{\scriptscriptstyle{(2)}}, W_{\scriptscriptstyle{(2)}})\\
    =&\Sigma\Big(X_{\scriptscriptstyle{(1)}}, \Sigma(Y_{\scriptscriptstyle{(1)}}, \Sigma(Z_{\scriptscriptstyle{(1)}}, W_{\scriptscriptstyle{(1)}})Z_{\scriptscriptstyle{(2)}}W_{\scriptscriptstyle{(2)}})Y_{\scriptscriptstyle{(2)}}(Z_{\scriptscriptstyle{(3)}}W_{\scriptscriptstyle{(3)}})\Big)\Phi(X_{\scriptscriptstyle{(2)}}, Y_{\scriptscriptstyle{(3)}}, Z_{\scriptscriptstyle{(4)}}W_{\scriptscriptstyle{(4)}})\\
    &\quad \Sigma^{-1}( \overline{\Sigma^{-1}(X_{\scriptscriptstyle{(4)}}, Y_{\scriptscriptstyle{(5)}})}\,X_{\scriptscriptstyle{(3)}}Y_{\scriptscriptstyle{(4)}},  \overline{\Sigma^{-1}(Z_{\scriptscriptstyle{(6)}}, W_{\scriptscriptstyle{(6)}})}\,Z_{\scriptscriptstyle{(5)}}W_{\scriptscriptstyle{(5)}})\\
    &\quad \Sigma(\Sigma(X_{\scriptscriptstyle{(5)}}, Y_{\scriptscriptstyle{(6)}})X_{\scriptscriptstyle{(6)}}Y_{\scriptscriptstyle{(7)}}, \Sigma(Z_{\scriptscriptstyle{(7)}}, W_{\scriptscriptstyle{(7)}})Z_{\scriptscriptstyle{(8)}}W_{\scriptscriptstyle{(8)}})\\
    &\quad \Phi(X_{\scriptscriptstyle{(7)}}Y_{\scriptscriptstyle{(8)}}, Z_{\scriptscriptstyle{(9)}}, W_{\scriptscriptstyle{(9)}})\Sigma^{-1}\Big( \overline{\Sigma^{-1}\big( \overline{\Sigma^{-1}(X_{\scriptscriptstyle{(10)}}, Y_{\scriptscriptstyle{(11)}})}\,X_{\scriptscriptstyle{(9)}}Y_{\scriptscriptstyle{(10)}}, Z_{\scriptscriptstyle{(11)}}\big)}(X_{\scriptscriptstyle{(8)}}Y_{\scriptscriptstyle{(9)}})Z_{\scriptscriptstyle{(10)}}, W_{\scriptscriptstyle{(10)}}\Big)\\
    =&\Sigma\Big(X_{\scriptscriptstyle{(1)}}, \Sigma(Y_{\scriptscriptstyle{(1)}}, \Sigma(Z_{\scriptscriptstyle{(1)}}, W_{\scriptscriptstyle{(1)}})Z_{\scriptscriptstyle{(2)}}W_{\scriptscriptstyle{(2)}})\,Y_{\scriptscriptstyle{(2)}}(Z_{\scriptscriptstyle{(3)}}W_{\scriptscriptstyle{(3)}})\Big)\Phi(X_{\scriptscriptstyle{(2)}}, Y_{\scriptscriptstyle{(3)}}, Z_{\scriptscriptstyle{(4)}}W_{\scriptscriptstyle{(4)}})\\
    &\quad \Phi(X_{\scriptscriptstyle{(3)}}Y_{\scriptscriptstyle{(4)}}, Z_{\scriptscriptstyle{(5)}}, W_{\scriptscriptstyle{(5)}})\Sigma^{-1}\Big(\overline{\Sigma^{-1}\big(\overline{\Sigma^{-1}(X_{\scriptscriptstyle{(6)}}, Y_{\scriptscriptstyle{(7)}})}\,X_{\scriptscriptstyle{(5)}}Y_{\scriptscriptstyle{(6)}}, Z_{\scriptscriptstyle{(7)}}\big)}(X_{\scriptscriptstyle{(4)}}Y_{\scriptscriptstyle{(5)}})Z_{\scriptscriptstyle{(6)}}, W_{\scriptscriptstyle{(6)}}\Big).
\end{align*}
On the other hand,
\begin{align*}
    \Phi^{\Sigma}(X_{\scriptscriptstyle{(1)}}&, \Phi^{\Sigma}(Y_{\scriptscriptstyle{(1)}}, Z_{\scriptscriptstyle{(1)}}, W_{\scriptscriptstyle{(1)}})Y_{\scriptscriptstyle{(2)}}\cdot_{\Sigma}Z_{\scriptscriptstyle{(2)}}, W_{\scriptscriptstyle{(2)}})\Phi^{\Sigma}(X_{\scriptscriptstyle{(2)}}, Y_{\scriptscriptstyle{(3)}}, Z_{\scriptscriptstyle{(3)}})\\
    =&\Sigma\Big(X_{\scriptscriptstyle{(1)}}, \Sigma(Y_{\scriptscriptstyle{(1)}}, \Sigma(Z_{\scriptscriptstyle{(1)}}, W_{\scriptscriptstyle{(1)}})\,Z_{\scriptscriptstyle{(2)}}W_{\scriptscriptstyle{(2)}})\Phi(Y_{\scriptscriptstyle{(2)}}, Z_{\scriptscriptstyle{(3)}}, W_{\scriptscriptstyle{(3)}})\Sigma^{-1}(\overline{\Sigma^{-1}(Y_{\scriptscriptstyle{(4)}}, Z_{\scriptscriptstyle{(5)}})}Y_{\scriptscriptstyle{(3)}}Z_{\scriptscriptstyle{(4)}}, W_{\scriptscriptstyle{(4)}})\\    &\quad \Sigma(\Sigma(Y_{\scriptscriptstyle{(5)}}, Z_{\scriptscriptstyle{(5)}})\,Y_{\scriptscriptstyle{(6)}} Z_{\scriptscriptstyle{(7)}}, W_{\scriptscriptstyle{(5)}})\,(Y_{\scriptscriptstyle{(7)}}Z_{\scriptscriptstyle{(8)}})W_{\scriptscriptstyle{(6)}}\Big)\Phi(X_{\scriptscriptstyle{(2)}}, Y_{\scriptscriptstyle{(8)}}Z_{\scriptscriptstyle{(9)}}, W_{\scriptscriptstyle{(7)}})\\
    &\quad \Sigma^{-1}\Big(\overline{\Sigma^{-1}(X_{\scriptscriptstyle{(4)}}, \overline{\Sigma^{-1}(Y_{\scriptscriptstyle{(11)}}, Z_{\scriptscriptstyle{(12)}})}\,Y_{\scriptscriptstyle{(10)}}Z_{\scriptscriptstyle{(11)}})}\,X_{\scriptscriptstyle{(3)}}(Y_{\scriptscriptstyle{(9)}}Z_{\scriptscriptstyle{(10)}}), W_{\scriptscriptstyle{(8)}}\Big)\\
    &\quad \Sigma(X_{\scriptscriptstyle{(5)}}, \Sigma(Y_{\scriptscriptstyle{(12)}}, Z_{\scriptscriptstyle{(13)}})Y_{\scriptscriptstyle{(13)}}Z_{\scriptscriptstyle{(14)}})\Phi(X_{\scriptscriptstyle{(6)}}, Y_{\scriptscriptstyle{(14)}}, Z_{\scriptscriptstyle{(15)}})\Sigma^{-1}(\overline{\Sigma^{-1}(X_{\scriptscriptstyle{(8)}}, Y_{\scriptscriptstyle{(16)}})}\,X_{\scriptscriptstyle{(7)}}Y_{\scriptscriptstyle{(15)}}, Z_{\scriptscriptstyle{(16)}})\\
    =&\Sigma\Big(X_{\scriptscriptstyle{(1)}}, \Sigma(Y_{\scriptscriptstyle{(1)}}, \Sigma(Z_{\scriptscriptstyle{(1)}}, W_{\scriptscriptstyle{(1)}})Z_{\scriptscriptstyle{(2)}}W_{\scriptscriptstyle{(2)}})\Phi(Y_{\scriptscriptstyle{(2)}}, Z_{\scriptscriptstyle{(3)}}, W_{\scriptscriptstyle{(3)}})(Y_{\scriptscriptstyle{(3)}}Z_{\scriptscriptstyle{(4)}})W_{\scriptscriptstyle{(4)}}\Big)\Phi(X_{\scriptscriptstyle{(2)}}, Y_{\scriptscriptstyle{(4)}}Z_{\scriptscriptstyle{(5)}}, W_{\scriptscriptstyle{(5)}})\\
    &\quad \Sigma^{-1}\Big(X_{\scriptscriptstyle{(3)}}(Y_{\scriptscriptstyle{(5)}}Z_{\scriptscriptstyle{(6)}}), W_{\scriptscriptstyle{(6)}}\Big) \Phi(X_{\scriptscriptstyle{(4)}}, Y_{\scriptscriptstyle{(6)}}, Z_{\scriptscriptstyle{(7)}})\Sigma^{-1}(\overline{\Sigma^{-1}(X_{\scriptscriptstyle{(6)}}, Y_{\scriptscriptstyle{(8)}})}\,X_{\scriptscriptstyle{(5)}}Y_{\scriptscriptstyle{(7)}}, Z_{\scriptscriptstyle{(8)}})\\
    =&\Sigma\Big(X_{\scriptscriptstyle{(1)}}, \Sigma(Y_{\scriptscriptstyle{(1)}}, \Sigma(Z_{\scriptscriptstyle{(1)}}, W_{\scriptscriptstyle{(1)}})Z_{\scriptscriptstyle{(2)}}W_{\scriptscriptstyle{(2)}})\overline{\Phi(Y_{\scriptscriptstyle{(3)}}, Z_{\scriptscriptstyle{(4)}}, W_{\scriptscriptstyle{(4)}})}\,Y_{\scriptscriptstyle{(2)}}(Z_{\scriptscriptstyle{(3)}}W_{\scriptscriptstyle{(3)}})\Big)\Phi(X_{\scriptscriptstyle{(2)}}, Y_{\scriptscriptstyle{(4)}}Z_{\scriptscriptstyle{(5)}}, W_{\scriptscriptstyle{(5)}})\\
    &\quad \Sigma^{-1}\Big(\overline{\Phi(X_{\scriptscriptstyle{(4)}}, Y_{\scriptscriptstyle{(6)}}, Z_{\scriptscriptstyle{(7)}})}\,X_{\scriptscriptstyle{(3)}}(Y_{\scriptscriptstyle{(5)}}Z_{\scriptscriptstyle{(6)}}), W_{\scriptscriptstyle{(6)}}\Big) \Sigma^{-1}(\overline{\Sigma^{-1}(X_{\scriptscriptstyle{(6)}}, Y_{\scriptscriptstyle{(8)}})}\,X_{\scriptscriptstyle{(5)}}Y_{\scriptscriptstyle{(7)}}, Z_{\scriptscriptstyle{(8)}})\\
    =&\Sigma\Big(X_{\scriptscriptstyle{(1)}}, \Sigma(Y_{\scriptscriptstyle{(1)}}, \Sigma(Z_{\scriptscriptstyle{(1)}}, W_{\scriptscriptstyle{(1)}})Z_{\scriptscriptstyle{(2)}}W_{\scriptscriptstyle{(2)}})\,Y_{\scriptscriptstyle{(2)}}(Z_{\scriptscriptstyle{(3)}}W_{\scriptscriptstyle{(3)}})\Big)\Phi(X_{\scriptscriptstyle{(2)}}, \Phi(Y_{\scriptscriptstyle{(3)}}, Z_{\scriptscriptstyle{(4)}}, W_{\scriptscriptstyle{(4)}})Y_{\scriptscriptstyle{(4)}}Z_{\scriptscriptstyle{(5)}}, W_{\scriptscriptstyle{(5)}})\\
    &\quad \Sigma^{-1}\Big(\Phi(X_{\scriptscriptstyle{(3)}}, Y_{\scriptscriptstyle{(5)}}, Z_{\scriptscriptstyle{(6)}})(X_{\scriptscriptstyle{(4)}}Y_{\scriptscriptstyle{(6)}})Z_{\scriptscriptstyle{(7)}}, W_{\scriptscriptstyle{(6)}}\Big) \Sigma^{-1}(\overline{\Sigma^{-1}(X_{\scriptscriptstyle{(6)}}, Y_{\scriptscriptstyle{(8)}})}\,X_{\scriptscriptstyle{(5)}}Y_{\scriptscriptstyle{(7)}}, Z_{\scriptscriptstyle{(8)}})\\
    =&\Sigma\Big(X_{\scriptscriptstyle{(1)}}, \Sigma(Y_{\scriptscriptstyle{(1)}}, \Sigma(Z_{\scriptscriptstyle{(1)}}, W_{\scriptscriptstyle{(1)}})Z_{\scriptscriptstyle{(2)}}W_{\scriptscriptstyle{(2)}})\,Y_{\scriptscriptstyle{(2)}}(Z_{\scriptscriptstyle{(3)}}W_{\scriptscriptstyle{(3)}})\Big)\Phi(X_{\scriptscriptstyle{(2)}}, \Phi(Y_{\scriptscriptstyle{(3)}}, Z_{\scriptscriptstyle{(4)}}, W_{\scriptscriptstyle{(4)}})Y_{\scriptscriptstyle{(4)}}Z_{\scriptscriptstyle{(5)}}, W_{\scriptscriptstyle{(5)}})\\
    &\quad \Phi(X_{\scriptscriptstyle{(3)}}, Y_{\scriptscriptstyle{(5)}}, Z_{\scriptscriptstyle{(6)}})\Sigma^{-1}\Big(\overline{\Sigma^{-1}(\overline{\Sigma^{-1}(X_{\scriptscriptstyle{(6)}}, Y_{\scriptscriptstyle{(8)}})}X_{\scriptscriptstyle{(5)}}Y_{\scriptscriptstyle{(7)}}, Z_{\scriptscriptstyle{(8)}})}\,(X_{\scriptscriptstyle{(4)}}Y_{\scriptscriptstyle{(6)}})Z_{\scriptscriptstyle{(7)}}, W_{\scriptscriptstyle{(6)}}\Big).
\end{align*}
These are equal by (3) for $\Phi$. \end{proof}

As a special case, given a usual bialgebroid $\CL$ (over $B$), we let $C^{2}(\cL, B)$ be the group of unital convolution-invertible maps $\Sigma\in \Hom_{B^{e}-}(\cL\otimes_{B^e}\cL, B)$, obeying $\Sigma(X, Y\overline{b})=\Sigma(X, Yb)$ in line with previous sections. Then 

\begin{cor}
For any $\Sigma\in C^{2}(\cL, B)$, we can construct a 3-cocycle $\partial\Sigma$,
\begin{align}\label{boundary map1}
    \partial\Sigma(X, Y, Z)=\Sigma(X_{\scriptscriptstyle{(1)}}, \Sigma(Y_{\scriptscriptstyle{(1)}}, Z_{\scriptscriptstyle{(1)}})Y_{\scriptscriptstyle{(2)}}Z_{\scriptscriptstyle{(2)}})\Sigma^{-1}(\overline{\Sigma^{-1}(X_{\scriptscriptstyle{(3)}}, Y_{\scriptscriptstyle{(4)}})}\,X_{\scriptscriptstyle{(2)}}Y_{\scriptscriptstyle{(3)}}, Z_{\scriptscriptstyle{(3)}}),
\end{align}
where $\Sigma^{-1}$ is the convolution-inverse of $\Sigma$, $X,Y$ and $Z$ are elements of $\cL$, making $\CL^\Sigma$ a coquasi-bialgebroid.
\end{cor}

In principle, this can be organised into a non-abelian cohomology theory, provided that, as in the usual theory of coquasi-Hopf algebras \cite{Ma:book}, one also keeps track of the changed product. One way to do this would be to consider a 3-cocycle as  pair consisting of the product and $\phi$ defined on a unital $B$-coring, modulo transformations as the theorem by unital convolution-invertible maps.

\section{\bf Constructions for coquasi-bialgebroids}\label{secCM}

Having established the general formalism, we now turn to constructions. The aim of this section is to show that natural examples of coquasi-bialgebroids arise from coquasi-bialgebras acting,
 in a possibly non-strict sense, on an associative algebra \(B\). The main construction is a coquasi analogue of the Connes--Moscovici bialgebroid on \(B\otimes H\otimes B\), where the
 associativity constraint is governed by both the associator of \(H\) and the twisting data measuring the failure of strict module-algebra associativity.

\subsection{Coquasi Connes-Moscovici bialgebroids}

We first recall the ordinary Connes--Moscovici-type bialgebroid, motivated by the
Hopf algebra of transverse geometry of Connes and Moscovici \cite{CM},
and its non-trivial \(2\)-cocycle version in the bialgebroid framework
\cite{schau4}. 
We then replace the Hopf algebra by a coquasi-bialgebra and weaken the module-algebra condition accordingly. The resulting object has the same underlying \(B\)-coalgebroid
 as in the classical case, but its multiplication is no longer strictly associative. Instead, its associativity is controlled by an explicit \(3\)-cocycle built from the associator of \(H\) and 
the twisting map \(\gamma\).

More precisely, given  a measuring  $\la$ of $H$ on $B$ and $\gamma: H\ot H\to B$ an unital convolution invertible map, such that
    \begin{align*}
        &h\la(g\la b)=\gamma(h\o, g\o)(h_{\scriptscriptstyle{(2)}}g_{\scriptscriptstyle{(2)}})\la b\gamma^{-1}(h_{\scriptscriptstyle{(3)}}, g_{\scriptscriptstyle{(3)}}),\quad 1_{H}\la b=b,\\
        &(\one{h}\triangleright\gamma(\one{g},\one{f}))\gamma(\two{h}, \two{g}\two{f})=\gamma(\one{h}, \one{g})\gamma(\two{h}\two{g}, f),
    \end{align*}
for all $h, g, f\in H$ and $b\in B$, then the underlying vector space $B\ot H\ot B$ is a left $B$ bialgebroid, with the $B$-coring structure:
    \[s(b)=b\ot 1\ot 1,\quad t(b)=1\ot 1\ot b,\quad \Delta(b\ot h\ot b')=b\ot h\o\ot 1\ot_{B} 1\ot h_{\scriptscriptstyle{(2)}}\ot b',\quad \varepsilon(b\ot h\ot b')=bb'\varepsilon(h).\]
Moreover, the algebra structure is
    \[(a\ot h\ot a')(b\ot g\ot b')=a(h\o\la b)\gamma(h_{\scriptscriptstyle{(2)}}, g\o)\ot h_{\scriptscriptstyle{(3)}}g_{\scriptscriptstyle{(2)}}\ot \gamma^{-1}(h_{\scriptscriptstyle{(4)}}, g_{\scriptscriptstyle{(3)}})(h_{\scriptscriptstyle{(5)}}\la b')a',\]
    for all $h, g\in H$ and $a, a', b, b'\in B$. In the case that $\gamma$ is trivial, i.e. $\gamma(h, g)=\varepsilon(hg)$ for any $h,g\in H$, we can see $B\ot H\ot B$ is a left and anti-left Hopf algebroid. More precisely,  \begin{align*}
        X_{+}\ot_{\BB}X_{-}:=\lambda^{-1}(X\di 1)=(b\ot h\o\ot 1)\ot_{\BB} (S(h\th)\la b'\ot S(h\t)\ot 1),
    \end{align*}
    and
    \begin{align*}
        X_{[+]}\ot_{B}X_{[-]}:=\mu^{-1}(1\di X)=(1\ot h\th\ot b')\ot_{B}(1\ot S^{-1}(h\t)\ot S^{-1}(h\o)\la b),
    \end{align*}
    for any $X=b\ot h\ot b'\in B\ot H\ot B$.  Indeed, we can see on the one hand,
    \begin{align*}
        X\o{}_{+}\ot_{B^{op}}X\o{}_{-}X\t=&(b\ot h\o\ot 1)\ot(1\ot1\ot S(h\t)\la b')\\
        =&(b\ot h\ot b')\ot (1\ot 1\ot 1),
    \end{align*}
    on the other hand
    \begin{align*}
        X_{+}\o\di X_{+}\t X_{-}=&(b\ot h\o\ot 1)\ot(b'\ot 1\ot 1)\\
        =&(b\ot h\ot b')\ot (1\ot 1\ot 1).
    \end{align*}
    The anti-left Hopf structure can be similarly proved.

We now relax the cocycle condition on \(\gamma\). More precisely, we allow \(H\) to be a coquasi-bialgebra and let \(B\) be an algebra equipped with a measuring by \(H\). Although a coquasi-bialgebra is a coquasi-bialgebroid in the special case \(B=k\), we recall the definition for use in the computations below. Quasi-Hopf algebras were introduced by Drinfeld \cite{Drinfeld1990},
 and we follow the coquasi conventions of \cite{Ma:book}.

\begin{defi}
    A coquasi-bialgebra \(H\) is a coalgebra equipped with a unital, not-necessarily associative algebra structure whose product and unit are coalgebra morphisms, together with a
 convolution-invertible linear map  $\phi: H\ot H\ot H\to k$, such that\begin{itemize}
        \item [(1)] $\phi(h, 1, g)=\varepsilon(hg)$.
        \item[(2)] $\phi(h\o, g\o, f\o l\o )\phi(h_{\scriptscriptstyle{(2)}}g_{\scriptscriptstyle{(2)}}, f_{\scriptscriptstyle{(2)}}, l_{\scriptscriptstyle{(2)}})=\phi(g\o , f\o , l\o )\phi(h\o , g_{\scriptscriptstyle{(2)}}f_{\scriptscriptstyle{(2)}}, l_{\scriptscriptstyle{(2)}})\phi(h_{\scriptscriptstyle{(2)}}, g_{\scriptscriptstyle{(3)}}, f_{\scriptscriptstyle{(3)}})$.
        \item [(3)] $\phi(h\o , g\o , f\o )(h_{\scriptscriptstyle{(2)}} g_{\scriptscriptstyle{(2)}}) f_{\scriptscriptstyle{(2)}}=h\o (g\o f\o )\phi(h_{\scriptscriptstyle{(2)}}, g_{\scriptscriptstyle{(2)}}, f_{\scriptscriptstyle{(2)}})$,
    \end{itemize}for any $h, g,f, l\in H$.
\end{defi}

We recall the standard notion of a measuring of a coalgebra on an algebra, in the
sense familiar from Hopf algebra actions on rings \cite{mont}.

\begin{defi}
    Let $(H, \phi)$ be a coquasi-bialgebra and $B$ a unital associative algebra. A map $\la: H\ot B\to B$ is a \textup{measuring} of $H$ on $B$ if $h\la(ab)=(h\o \la a)(h_{\scriptscriptstyle{(2)}}\la b)$ and $h\la 1_{B}=\varepsilon(h)1_{B}$ for all $h\in H$ and $a,b \in B$.
\end{defi}

We call $B$ a strict left $H$-module algebra if $\la$ is a measuring such that $(hg)\la b=h\la(g\la b)$ and $1_{H}\la b=b$. This condition is rather strong in the coquasi setting; the following theorem uses a weaker, non-strict version of an \(H\)-module algebra. The crossed-product and cleft-extension viewpoint for coquasi-Hopf algebras was
developed in \cite{Balan2010,Balan2008}, and our formulas may be viewed as bialgebroid
analogues of this construction over a noncommutative base.

\begin{thm}\label{thm. coquasi CM bialgebroid}
    Let $\la: H\ot B\to B$ be a measuring of a coquasi-bialgebra $(H, \phi)$ on an unital associative algebra $B$. If $\gamma: H\ot H\to B$ is an unital convolution invertible map such that
    \[h\la(g\la b)=\gamma(h\o , g\o )((h_{\scriptscriptstyle{(2)}}g_{\scriptscriptstyle{(2)}})\la b)\gamma^{-1}(h_{\scriptscriptstyle{(3)}}, g_{\scriptscriptstyle{(3)}}),\quad 1_{H}\la b=b,\]
    for any $h, g\in H$ and $b\in B$, then the underlying vector space $B\ot H\ot B$ is a left $B$-coquasi-bialgebroid, with the $B$-coalgebroid structure:
    \[s(b)=b\ot 1\ot 1\quad t(b)=1\ot 1\ot b\quad \Delta(b\ot h\ot b')=b\ot h\o \ot 1\ot_{B} 1\ot h_{\scriptscriptstyle{(2)}}\ot b'\quad \varepsilon(b\ot h\ot b')=bb'\varepsilon(h).\]
    And the (nonassociative unital) algebra structure:
    \[(a\ot h\ot a')(b\ot g\ot b')=a(h\o \la b)\gamma(h_{\scriptscriptstyle{(2)}}, g\o )\ot h_{\scriptscriptstyle{(3)}}g_{\scriptscriptstyle{(2)}}\ot \gamma^{-1}(h_{\scriptscriptstyle{(4)}}, g_{\scriptscriptstyle{(3)}})(h_{\scriptscriptstyle{(5)}}\la b')a',\]
    for any $h, g\in H$ and $a, a', b, b'\in B$.
    The `invertible normalised 3-cocycle' $\Phi^{\gamma}$ is given by
    \begin{align*}
        \Phi^{\gamma}(X&, Y, Z)=a(h\o \la b)(h_{\scriptscriptstyle{(2)}}\la(g\o \la c))(h_{\scriptscriptstyle{(3)}}\la\gamma(g_{\scriptscriptstyle{(2)}}, f\o ))\gamma(h_{\scriptscriptstyle{(4)}}, g_{\scriptscriptstyle{(3)}}f_{\scriptscriptstyle{(2)}})\\
    &\phi(h_{\scriptscriptstyle{(5)}}, g_{\scriptscriptstyle{(4)}}, f_{\scriptscriptstyle{(3)}})\gamma^{-1}(h_{\scriptscriptstyle{(6)}}g_{\scriptscriptstyle{(5)}}, f_{\scriptscriptstyle{(4)}})\gamma^{-1}(h_{\scriptscriptstyle{(7)}}, g_{\scriptscriptstyle{(6)}})(h_{\scriptscriptstyle{(8)}}\la(g_{\scriptscriptstyle{(7)}}\la c'))(h_{\scriptscriptstyle{(9)}}\la b')a',
    \end{align*}
  for $X=a\ot h\ot a'$, $Y=b\ot g\ot b'$ and $Z=c\ot f\ot c'$.
\end{thm}
\begin{proof}
    The \(B\)-coalgebroid structure on \(B\otimes H\otimes B\) is the same as that of the usual Connes--Moscovici bialgebroid. The coproduct is an algebra map; indeed,
    \begin{align*}
        \Delta(XY)=&a(h\o \la b)\gamma(h_{\scriptscriptstyle{(2)}}, g\o )\ot h_{\scriptscriptstyle{(3)}}g_{\scriptscriptstyle{(2)}}\ot 1\ot_{B} 1\ot h_{\scriptscriptstyle{(4)}}g_{\scriptscriptstyle{(3)}} \ot\gamma^{-1}(h_{\scriptscriptstyle{(5)}}, g_{\scriptscriptstyle{(4)}})(h_{\scriptscriptstyle{(6)}}\la b')a'\\
        =&a(h\o \la b)\gamma(h_{\scriptscriptstyle{(2)}}, g\o )\ot h_{\scriptscriptstyle{(3)}}g_{\scriptscriptstyle{(2)}}\ot \gamma^{-1}(h_{\scriptscriptstyle{(4)}}, g_{\scriptscriptstyle{(3)}})\ot_{B} \gamma(h_{\scriptscriptstyle{(5)}}, g_{\scriptscriptstyle{(4)}})\ot h_{\scriptscriptstyle{(6)}}g_{\scriptscriptstyle{(5)}}\ot \gamma^{-1}(h_{\scriptscriptstyle{(7)}}, g_{\scriptscriptstyle{(6)}})(h_{\scriptscriptstyle{(8)}}\la b')a'\\
        =&\Delta(X)\Delta(Y).
    \end{align*}
We can also see
\begin{align*}
    X\o \,\overline{b}\ot_{B}X_{\scriptscriptstyle{(2)}}=&a\ot h\o \ot h_{\scriptscriptstyle{(2)}}\la b\ot_{B} 1\ot h_{\scriptscriptstyle{(3)}}\ot a'\\
    =&a\ot h\o \ot 1 \ot_{B} h_{\scriptscriptstyle{(2)}}\la b\ot h_{\scriptscriptstyle{(3)}}\ot a'\\
=&X\o \ot_{B}X_{\scriptscriptstyle{(2)}}\,b.
\end{align*}
For the counit, we have
\begin{align*}    \varepsilon(XY)=a(h\la(bb'))a'\varepsilon(g)=\varepsilon(X\,\varepsilon(Y))=\varepsilon(X \overline{\varepsilon(Y)}).
\end{align*}
It is straightforward to verify that \(\Phi^\gamma\) is \(B\)-bilinear, well defined over \(\otimes_{B^e}\), and unital. Moreover,
\begin{align*}
    \Phi^{\gamma}(X, Y, Z\,d)=&a(h\o \la b)(h_{\scriptscriptstyle{(2)}}\la(g\o \la c))(h_{\scriptscriptstyle{(3)}}\la(g_{\scriptscriptstyle{(2)}}\la(f\o \la d)))(h_{\scriptscriptstyle{(4)}}\la\gamma(g_{\scriptscriptstyle{(3)}}, f_{\scriptscriptstyle{(2)}}))\gamma(h_{\scriptscriptstyle{(5)}}, g_{\scriptscriptstyle{(4)}}f_{\scriptscriptstyle{(3)}})\\
    &\quad \phi(h_{\scriptscriptstyle{(6)}}, g_{\scriptscriptstyle{(5)}}, f_{\scriptscriptstyle{(4)}})\gamma^{-1}(h_{\scriptscriptstyle{(7)}}g_{\scriptscriptstyle{(6)}}, f_{\scriptscriptstyle{(5)}})\gamma^{-1}(h_{\scriptscriptstyle{(8)}}, g_{\scriptscriptstyle{(7)}})(h_{\scriptscriptstyle{(9)}}\la(g_{\scriptscriptstyle{(8)}}\la c'))(h_{\scriptscriptstyle{(10)}}\la b')a'\\
    =&a(h\o \la b)(h_{\scriptscriptstyle{(2)}}\la(g\o \la c))(h_{\scriptscriptstyle{(3)}}\la\gamma(g_{\scriptscriptstyle{(2)}}, f\o ))\gamma(h_{\scriptscriptstyle{(4)}}, g_{\scriptscriptstyle{(3)}}f_{\scriptscriptstyle{(2)}})((h_{\scriptscriptstyle{(5)}}(g_{\scriptscriptstyle{(4)}}f_{\scriptscriptstyle{(3)}}))\la d)\\
    &\quad \phi(h_{\scriptscriptstyle{(6)}}, g_{\scriptscriptstyle{(5)}}, f_{\scriptscriptstyle{(4)}})\gamma^{-1}(h_{\scriptscriptstyle{(7)}}g_{\scriptscriptstyle{(6)}}, f_{\scriptscriptstyle{(5)}})\gamma^{-1}(h_{\scriptscriptstyle{(8)}}, g_{\scriptscriptstyle{(7)}})(h_{\scriptscriptstyle{(9)}}\la(g_{\scriptscriptstyle{(8)}}\la c'))(h_{\scriptscriptstyle{(10)}}\la b')a'\\
    =&a(h\o \la b)(h_{\scriptscriptstyle{(2)}}\la(g\o \la c))(h_{\scriptscriptstyle{(3)}}\la\gamma(g_{\scriptscriptstyle{(2)}}, f\o ))\gamma(h_{\scriptscriptstyle{(4)}}, g_{\scriptscriptstyle{(3)}}f_{\scriptscriptstyle{(2)}})\phi(h_{\scriptscriptstyle{(5)}}, g_{\scriptscriptstyle{(4)}}, f_{\scriptscriptstyle{(3)}})\\
    &\quad (((h_{\scriptscriptstyle{(6)}}g_{\scriptscriptstyle{(5)}})f_{\scriptscriptstyle{(4)}})\la d) \gamma^{-1}(h_{\scriptscriptstyle{(7)}}g_{\scriptscriptstyle{(6)}}, f_{\scriptscriptstyle{(5)}})\gamma^{-1}(h_{\scriptscriptstyle{(8)}}, g_{\scriptscriptstyle{(7)}})(h_{\scriptscriptstyle{(9)}}\la(g_{\scriptscriptstyle{(8)}}\la c'))(h_{\scriptscriptstyle{(10)}}\la b')a'\\
    =&a(h\o \la b)(h_{\scriptscriptstyle{(2)}}\la(g\o \la c))(h_{\scriptscriptstyle{(3)}}\la\gamma(g_{\scriptscriptstyle{(2)}}, f\o ))\gamma(h_{\scriptscriptstyle{(4)}}, g_{\scriptscriptstyle{(3)}}f_{\scriptscriptstyle{(2)}})\phi(h_{\scriptscriptstyle{(5)}}, g_{\scriptscriptstyle{(4)}}, f_{\scriptscriptstyle{(3)}})\\
    &\quad \gamma^{-1}(h_{\scriptscriptstyle{(6)}}g_{\scriptscriptstyle{(5)}}, f_{\scriptscriptstyle{(4)}})\gamma^{-1}(h_{\scriptscriptstyle{(7)}}, g_{\scriptscriptstyle{(6)}})(h_{\scriptscriptstyle{(8)}}\la(g_{\scriptscriptstyle{(7)}}\la(f_{\scriptscriptstyle{(5)}}\la d)))(h_{\scriptscriptstyle{(9)}}\la(g_{\scriptscriptstyle{(8)}}\la c'))(h_{\scriptscriptstyle{(10)}}\la b')a'\\
    =&\Phi^{\gamma}(X, Y, Z\,\overline{d})
\end{align*}
    We can also see the convolution inverse of $\Phi^{\gamma}$ is
    \begin{align*}
        (\Phi^{\gamma})^{-1}(X&, Y, Z)=a(h\o \la b)(h_{\scriptscriptstyle{(2)}}\la(g\o \la c))\gamma(h_{\scriptscriptstyle{(3)}}, g_{\scriptscriptstyle{(2)}})\gamma(h_{\scriptscriptstyle{(4)}}g_{\scriptscriptstyle{(3)}}, f\o )\\
    &\phi^{-1}(h_{\scriptscriptstyle{(5)}}, g_{\scriptscriptstyle{(4)}}, f_{\scriptscriptstyle{(2)}})\gamma^{-1}(h_{\scriptscriptstyle{(6)}}, g_{\scriptscriptstyle{(5)}}f_{\scriptscriptstyle{(3)}})h_{\scriptscriptstyle{(7)}}\la\gamma^{-1}(g_{\scriptscriptstyle{(6)}}, f_{\scriptscriptstyle{(4)}})(h_{\scriptscriptstyle{(8)}}\la(g_{\scriptscriptstyle{(7)}}\la c'))(h_{\scriptscriptstyle{(9)}}\la b')a',
    \end{align*}
and
\begin{align*}
    \Phi^{\gamma}(X\o &, Y\o , Z\o )(X_{\scriptscriptstyle{(2)}}Y_{\scriptscriptstyle{(2)}})Z_{\scriptscriptstyle{(2)}}\\
    =&a(h\o \la b)(h_{\scriptscriptstyle{(2)}}\la(g\o \la c))(h_{\scriptscriptstyle{(3)}}\la\gamma(g_{\scriptscriptstyle{(2)}}, f\o ))\gamma(h_{\scriptscriptstyle{(4)}}, g_{\scriptscriptstyle{(3)}}f_{\scriptscriptstyle{(2)}})\phi(h_{\scriptscriptstyle{(5)}}, g_{\scriptscriptstyle{(4)}}, f_{\scriptscriptstyle{(3)}})\gamma^{-1}(h_{\scriptscriptstyle{(6)}}g_{\scriptscriptstyle{(5)}}, f_{\scriptscriptstyle{(4)}})\\
    &\quad\gamma^{-1}(h_{\scriptscriptstyle{(7)}}, g_{\scriptscriptstyle{(6)}})\gamma(h_{\scriptscriptstyle{(8)}}, g_{\scriptscriptstyle{(7)}})\gamma(h_{\scriptscriptstyle{(9)}}g_{\scriptscriptstyle{(8)}}, f_{\scriptscriptstyle{(5)}})\ot (h_{\scriptscriptstyle{(10)}}g_{\scriptscriptstyle{(9)}})f_{\scriptscriptstyle{(6)}}\ot\\
    &\quad \gamma^{-1}(h_{\scriptscriptstyle{(11)}}g_{\scriptscriptstyle{(10)}}, f_{\scriptscriptstyle{(7)}})\gamma^{-1}(h_{\scriptscriptstyle{(12)}}, g_{\scriptscriptstyle{(11)}})(h_{\scriptscriptstyle{(13)}}\la(g_{\scriptscriptstyle{(12)}}\la c'))(h_{\scriptscriptstyle{(14)}}\la b')a'\\
    =&a(h\o \la b)(h_{\scriptscriptstyle{(2)}}\la(g\o \la c))(h_{\scriptscriptstyle{(3)}}\la\gamma(g_{\scriptscriptstyle{(2)}}, f\o ))\gamma(h_{\scriptscriptstyle{(4)}}, g_{\scriptscriptstyle{(3)}}f_{\scriptscriptstyle{(2)}})\phi(h_{\scriptscriptstyle{(5)}}, g_{\scriptscriptstyle{(4)}}, f_{\scriptscriptstyle{(3)}})\\
    &\quad\ot (h_{\scriptscriptstyle{(6)}}g_{\scriptscriptstyle{(5)}})f_{\scriptscriptstyle{(4)}}\ot\gamma^{-1}(h_{\scriptscriptstyle{(7)}}g_{\scriptscriptstyle{(6)}}, f_{\scriptscriptstyle{(5)}})\gamma^{-1}(h_{\scriptscriptstyle{(8)}}, g_{\scriptscriptstyle{(7)}})(h_{\scriptscriptstyle{(9)}}\la(g_{\scriptscriptstyle{(8)}}\la c'))(h_{\scriptscriptstyle{(10)}}\la b')a'\\
    =&a(h\o \la b)(h_{\scriptscriptstyle{(2)}}\la(g\o \la c))(h_{\scriptscriptstyle{(3)}}\la\gamma(g_{\scriptscriptstyle{(2)}}, f\o ))\gamma(h_{\scriptscriptstyle{(4)}}, g_{\scriptscriptstyle{(3)}}f_{\scriptscriptstyle{(2)}})\ot h_{\scriptscriptstyle{(5)}}(g_{\scriptscriptstyle{(4)}}f_{\scriptscriptstyle{(3)}})\ot\\\
    &\quad\phi(h_{\scriptscriptstyle{(6)}}, g_{\scriptscriptstyle{(5)}}, f_{\scriptscriptstyle{(4)}})\gamma^{-1}(h_{\scriptscriptstyle{(7)}}g_{\scriptscriptstyle{(6)}}, f_{\scriptscriptstyle{(5)}})\gamma^{-1}(h_{\scriptscriptstyle{(8)}}, g_{\scriptscriptstyle{(7)}})(h_{\scriptscriptstyle{(9)}}\la(g_{\scriptscriptstyle{(8)}}\la c'))(h_{\scriptscriptstyle{(10)}}\la b')a'\\
    =&\overline{\Phi^{\gamma}(X_{\scriptscriptstyle{(2)}}, Y_{\scriptscriptstyle{(2)}}, Z_{\scriptscriptstyle{(2)}})}\,X\o (Y\o Z\o ).
\end{align*}
Finally, let $W=d\ot l\ot d'$, we can see, on the one hand
\begin{align*}
     \Phi^{\gamma}(X\o &, Y\o , Z\o W\o )\Phi^{\gamma}(X_{\scriptscriptstyle{(2)}}Y_{\scriptscriptstyle{(2)}}, Z_{\scriptscriptstyle{(2)}}, W_{\scriptscriptstyle{(2)}})\\
     =&a(h\o \la b)(h_{\scriptscriptstyle{(2)}}\la(g\o \la c))(h_{\scriptscriptstyle{(3)}}\la(g_{\scriptscriptstyle{(2)}}\la(f\o \la d)))(h_{\scriptscriptstyle{(4)}}\la(g_{\scriptscriptstyle{(3)}}\la \gamma(f_{\scriptscriptstyle{(2)}}, l\o )))(h_{\scriptscriptstyle{(5)}}\la\gamma(g_{\scriptscriptstyle{(4)}}, f_{\scriptscriptstyle{(3)}}l_{\scriptscriptstyle{(2)}}))\\
     &\quad\gamma(h_{\scriptscriptstyle{(6)}}, g_{\scriptscriptstyle{(5)}}(f_{\scriptscriptstyle{(4)}}l_{\scriptscriptstyle{(3)}}))\phi(h_{\scriptscriptstyle{(7)}}, g_{\scriptscriptstyle{(6)}}, f_{\scriptscriptstyle{(5)}}l_{\scriptscriptstyle{(4)}})\gamma^{-1}(h_{\scriptscriptstyle{(8)}}g_{\scriptscriptstyle{(7)}}, f_{\scriptscriptstyle{(6)}}l_{\scriptscriptstyle{(5)}})\gamma^{-1}(h_{\scriptscriptstyle{(9)}}, g_{\scriptscriptstyle{(8)}})(h_{\scriptscriptstyle{(10)}}\la(g_{\scriptscriptstyle{(9)}}\la\gamma^{-1}(f_{\scriptscriptstyle{(7)}}, l_{\scriptscriptstyle{(6)}})))\\
     &\quad\gamma(h_{\scriptscriptstyle{(11)}}, g_{\scriptscriptstyle{(10)}})((h_{\scriptscriptstyle{(12)}}g_{\scriptscriptstyle{(11)}})\la\gamma(f_{\scriptscriptstyle{(8)}}, l_{\scriptscriptstyle{(7)}}))\gamma(h_{\scriptscriptstyle{(13)}}g_{\scriptscriptstyle{(12)}}, f_{\scriptscriptstyle{(9)}}l_{\scriptscriptstyle{(8)}})\phi(h_{\scriptscriptstyle{(14)}}g_{\scriptscriptstyle{(13)}}, f_{\scriptscriptstyle{(10)}}, l_{\scriptscriptstyle{(9)}})\gamma^{-1}((h_{\scriptscriptstyle{(15)}}g_{\scriptscriptstyle{(14)}})f_{\scriptscriptstyle{(11)}}, l_{\scriptscriptstyle{(10)}})\\
     &\quad \gamma^{-1}(h_{\scriptscriptstyle{(16)}}g_{\scriptscriptstyle{(15)}}, f_{\scriptscriptstyle{(12)}})((h_{\scriptscriptstyle{(17)}}g_{\scriptscriptstyle{(16)}})\la(f_{\scriptscriptstyle{(13)}}\la d'))((h_{\scriptscriptstyle{(18)}}g_{\scriptscriptstyle{(17)}})\la c')\gamma^{-1}(h_{\scriptscriptstyle{(19)}}, g_{\scriptscriptstyle{(18)}})(h_{\scriptscriptstyle{(20)}}\la b')a'\\
     =&a(h\o \la b)(h_{\scriptscriptstyle{(2)}}\la(g\o \la c))(h_{\scriptscriptstyle{(3)}}\la(g_{\scriptscriptstyle{(2)}}\la(f\o \la d)))(h_{\scriptscriptstyle{(4)}}\la(g_{\scriptscriptstyle{(3)}}\la \gamma(f_{\scriptscriptstyle{(2)}}, l\o )))(h_{\scriptscriptstyle{(5)}}\la\gamma(g_{\scriptscriptstyle{(4)}}, f_{\scriptscriptstyle{(3)}}l_{\scriptscriptstyle{(2)}}))\\
     &\quad\gamma(h_{\scriptscriptstyle{(6)}}, g_{\scriptscriptstyle{(5)}}(f_{\scriptscriptstyle{(4)}}l_{\scriptscriptstyle{(3)}}))\phi(h_{\scriptscriptstyle{(7)}}, g_{\scriptscriptstyle{(6)}}, f_{\scriptscriptstyle{(5)}}l_{\scriptscriptstyle{(4)}})\phi(h_{\scriptscriptstyle{(8)}}g_{\scriptscriptstyle{(7)}}, f_{\scriptscriptstyle{(6)}}, l_{\scriptscriptstyle{(5)}})\gamma^{-1}((h_{\scriptscriptstyle{(9)}}g_{\scriptscriptstyle{(8)}})f_{\scriptscriptstyle{(7)}}, l_{\scriptscriptstyle{(6)}})\\
     &\quad \gamma^{-1}(h_{\scriptscriptstyle{(10)}}g_{\scriptscriptstyle{(9)}}, f_{\scriptscriptstyle{(8)}})((h_{\scriptscriptstyle{(11)}}g_{\scriptscriptstyle{(10)}})\la(f_{\scriptscriptstyle{(9)}}\la d'))((h_{\scriptscriptstyle{(12)}}g_{\scriptscriptstyle{(11)}})\la c')\gamma^{-1}(h_{\scriptscriptstyle{(13)}}, g_{\scriptscriptstyle{(12)}})(h_{\scriptscriptstyle{(14)}}\la b')a'\\
     =&a(h\o \la b)(h_{\scriptscriptstyle{(2)}}\la(g\o \la c))(h_{\scriptscriptstyle{(3)}}\la(g_{\scriptscriptstyle{(2)}}\la(f\o \la d)))(h_{\scriptscriptstyle{(4)}}\la(g_{\scriptscriptstyle{(3)}}\la \gamma(f_{\scriptscriptstyle{(2)}}, l\o )))(h_{\scriptscriptstyle{(5)}}\la\gamma(g_{\scriptscriptstyle{(4)}}, f_{\scriptscriptstyle{(3)}}l_{\scriptscriptstyle{(2)}}))\\
     &\quad\gamma(h_{\scriptscriptstyle{(6)}}, g_{\scriptscriptstyle{(5)}}(f_{\scriptscriptstyle{(4)}}l_{\scriptscriptstyle{(3)}}))\phi(h_{\scriptscriptstyle{(7)}}, g_{\scriptscriptstyle{(6)}}, f_{\scriptscriptstyle{(5)}}l_{\scriptscriptstyle{(4)}})\phi(h_{\scriptscriptstyle{(8)}}g_{\scriptscriptstyle{(7)}}, f_{\scriptscriptstyle{(6)}}, l_{\scriptscriptstyle{(5)}})\gamma^{-1}((h_{\scriptscriptstyle{(9)}}g_{\scriptscriptstyle{(8)}})f_{\scriptscriptstyle{(7)}}, l_{\scriptscriptstyle{(6)}})\\
     &\quad \gamma^{-1}(h_{\scriptscriptstyle{(10)}}g_{\scriptscriptstyle{(9)}}, f_{\scriptscriptstyle{(8)}})\gamma^{-1}(h_{\scriptscriptstyle{(11)}}, g_{\scriptscriptstyle{(10)}})(h_{\scriptscriptstyle{(12)}}\la(g_{\scriptscriptstyle{(11)}}\la(f_{\scriptscriptstyle{(9)}}\la d')))(h_{\scriptscriptstyle{(13)}}\la(g_{\scriptscriptstyle{(12)}}\la c'))(h_{\scriptscriptstyle{(14)}}\la b')a'.
\end{align*}
On the other hand,
\begin{align*}
    \Phi^{\gamma}(X\o &, s(\Phi^{\gamma}(Y\o , Z\o , W\o ))Y_{\scriptscriptstyle{(2)}}Z_{\scriptscriptstyle{(2)}}, W_{\scriptscriptstyle{(2)}})\Phi^{\gamma}(X_{\scriptscriptstyle{(2)}}, Y_{\scriptscriptstyle{(3)}}, Z_{\scriptscriptstyle{(3)}})\\    =&a(h\o \la b)(h_{\scriptscriptstyle{(2)}}\la(g\o \la c))(h_{\scriptscriptstyle{(3)}}\la(g_{\scriptscriptstyle{(2)}}\la(f\o \la d)))(h_{\scriptscriptstyle{(4)}}\la(g_{\scriptscriptstyle{(3)}}\la\gamma(f_{\scriptscriptstyle{(2)}}, l\o )))(h_{\scriptscriptstyle{(5)}}\la\gamma(g_{\scriptscriptstyle{(4)}}, f_{\scriptscriptstyle{(3)}}l_{\scriptscriptstyle{(2)}}))\\
    &\quad(h_{\scriptscriptstyle{(6)}}\la\phi(g_{\scriptscriptstyle{(5)}}, f_{\scriptscriptstyle{(4)}}, l_{\scriptscriptstyle{(3)}}))(h_{\scriptscriptstyle{(7)}}\la\gamma^{-1}(g_{\scriptscriptstyle{(6)}}f_{\scriptscriptstyle{(5)}}, l_{\scriptscriptstyle{(4)}})) (h_{\scriptscriptstyle{(8)}}\la\gamma^{-1}(g_{\scriptscriptstyle{(7)}}, f_{\scriptscriptstyle{(6)}}))(h_{\scriptscriptstyle{(9)}}\la \gamma(g_{\scriptscriptstyle{(8)}}, f_{\scriptscriptstyle{(7)}}))\\
    &\quad(h_{\scriptscriptstyle{(10)}}\la\gamma(g_{\scriptscriptstyle{(9)}}f_{\scriptscriptstyle{(8)}}, l_{\scriptscriptstyle{(5)}}))\gamma(h_{\scriptscriptstyle{(11)}}, (g_{\scriptscriptstyle{(10)}}f_{\scriptscriptstyle{(9)}})l_{\scriptscriptstyle{(6)}})\phi(h_{\scriptscriptstyle{(12)}}, g_{\scriptscriptstyle{(11)}}f_{\scriptscriptstyle{(10)}}, l_{\scriptscriptstyle{(7)}})\gamma^{-1}(h_{\scriptscriptstyle{(13)}}(g_{\scriptscriptstyle{(12)}}f_{\scriptscriptstyle{(11)}}), l_{\scriptscriptstyle{(8)}})\\
    &\quad\gamma^{-1}(h_{\scriptscriptstyle{(14)}}, g_{\scriptscriptstyle{(13)}}f_{\scriptscriptstyle{(12)}}) (h_{\scriptscriptstyle{(15)}}\la((g_{\scriptscriptstyle{(14)}}f_{\scriptscriptstyle{(13)}})\la d'))(h_{\scriptscriptstyle{(16)}}\la \gamma^{-1}(g_{\scriptscriptstyle{(15)}}, f_{\scriptscriptstyle{(14)}}))(h_{\scriptscriptstyle{(17)}}\la\gamma(g_{\scriptscriptstyle{(16)}}, f_{\scriptscriptstyle{(15)}}))\\
    &\quad\gamma(h_{\scriptscriptstyle{(18)}}, g_{\scriptscriptstyle{(17)}}f_{\scriptscriptstyle{(16)}})\phi(h_{\scriptscriptstyle{(19)}}, g_{\scriptscriptstyle{(18)}}, f_{\scriptscriptstyle{(17)}})\gamma^{-1}(h_{\scriptscriptstyle{(20)}}g_{\scriptscriptstyle{(19)}}, f_{\scriptscriptstyle{(18)}})\gamma^{-1}(h_{\scriptscriptstyle{(21)}}, g_{\scriptscriptstyle{(20)}})(h_{\scriptscriptstyle{(22)}}\la(g_{\scriptscriptstyle{(21)}}\la c'))(h_{\scriptscriptstyle{(23)}}\la b')a'\\
     =&a(h\o \la b)(h_{\scriptscriptstyle{(2)}}\la(g\o \la c))(h_{\scriptscriptstyle{(3)}}\la(g_{\scriptscriptstyle{(2)}}\la(f\o \la d)))(h_{\scriptscriptstyle{(4)}}\la(g_{\scriptscriptstyle{(3)}}\la\gamma(f_{\scriptscriptstyle{(2)}}, l\o )))(h_{\scriptscriptstyle{(5)}}\la\gamma(g_{\scriptscriptstyle{(4)}}, f_{\scriptscriptstyle{(3)}}l_{\scriptscriptstyle{(2)}}))\\
    &\quad(h_{\scriptscriptstyle{(6)}}\la\phi(g_{\scriptscriptstyle{(5)}}, f_{\scriptscriptstyle{(4)}}, l_{\scriptscriptstyle{(3)}}))\gamma(h_{\scriptscriptstyle{(7)}}, (g_{\scriptscriptstyle{(6)}}f_{\scriptscriptstyle{(5)}})l_{\scriptscriptstyle{(4)}})\phi(h_{\scriptscriptstyle{(8)}}, g_{\scriptscriptstyle{(7)}}f_{\scriptscriptstyle{(6)}}, l_{\scriptscriptstyle{(5)}})\gamma^{-1}(h_{\scriptscriptstyle{(9)}}(g_{\scriptscriptstyle{(8)}}f_{\scriptscriptstyle{(7)}}), l_{\scriptscriptstyle{(6)}})\\
    &\quad\gamma^{-1}(h_{\scriptscriptstyle{(10)}}, g_{\scriptscriptstyle{(9)}}f_{\scriptscriptstyle{(8)}})(h_{\scriptscriptstyle{(11)}}\la((g_{\scriptscriptstyle{(10)}}f_{\scriptscriptstyle{(9)}})\la d'))\gamma(h_{\scriptscriptstyle{(12)}}, g_{\scriptscriptstyle{(11)}}f_{\scriptscriptstyle{(10)}})\phi(h_{\scriptscriptstyle{(13)}}, g_{\scriptscriptstyle{(12)}}, f_{\scriptscriptstyle{(11)}})\gamma^{-1}(h_{\scriptscriptstyle{(14)}}g_{\scriptscriptstyle{(13)}}, f_{\scriptscriptstyle{(12)}})\\
    &\quad\gamma^{-1}(h_{\scriptscriptstyle{(15)}}, g_{\scriptscriptstyle{(14)}})(h_{\scriptscriptstyle{(16)}}\la(g_{\scriptscriptstyle{(15)}}\la c'))(h_{\scriptscriptstyle{(17)}}\la b')a'\\
     =&a(h\o \la b)(h_{\scriptscriptstyle{(2)}}\la(g\o \la c))(h_{\scriptscriptstyle{(3)}}\la(g_{\scriptscriptstyle{(2)}}\la(f\o \la d)))(h_{\scriptscriptstyle{(4)}}\la(g_{\scriptscriptstyle{(3)}}\la\gamma(f_{\scriptscriptstyle{(2)}}, l\o )))(h_{\scriptscriptstyle{(5)}}\la\gamma(g_{\scriptscriptstyle{(4)}}, f_{\scriptscriptstyle{(3)}}l_{\scriptscriptstyle{(2)}}))\\
    &\quad(h_{\scriptscriptstyle{(6)}}\la\phi(g_{\scriptscriptstyle{(5)}}, f_{\scriptscriptstyle{(4)}}, l_{\scriptscriptstyle{(3)}}))\gamma(h_{\scriptscriptstyle{(7)}}, (g_{\scriptscriptstyle{(6)}}f_{\scriptscriptstyle{(5)}})l_{\scriptscriptstyle{(4)}})\phi(h_{\scriptscriptstyle{(8)}}, g_{\scriptscriptstyle{(7)}}f_{\scriptscriptstyle{(6)}}, l_{\scriptscriptstyle{(5)}})\gamma^{-1}(h_{\scriptscriptstyle{(9)}}(g_{\scriptscriptstyle{(8)}}f_{\scriptscriptstyle{(7)}}), l_{\scriptscriptstyle{(6)}})\\
    &\quad\gamma^{-1}(h_{\scriptscriptstyle{(10)}}, g_{\scriptscriptstyle{(9)}}f_{\scriptscriptstyle{(8)}})\gamma(h_{\scriptscriptstyle{(11)}}, g_{\scriptscriptstyle{(10)}}f_{\scriptscriptstyle{(9)}})\phi(h_{\scriptscriptstyle{(12)}}, g_{\scriptscriptstyle{(11)}}, f_{\scriptscriptstyle{(10)}})\gamma^{-1}(h_{\scriptscriptstyle{(13)}}g_{\scriptscriptstyle{(12)}}, f_{\scriptscriptstyle{(11)}})\\
    &\quad\gamma^{-1}(h_{\scriptscriptstyle{(14)}}, g_{\scriptscriptstyle{(13)}})(h_{\scriptscriptstyle{(15)}}\la(g_{\scriptscriptstyle{(14)}}\la(f_{\scriptscriptstyle{(12)}}\la d')))(h_{\scriptscriptstyle{(16)}}\la(g_{\scriptscriptstyle{(15)}}\la c'))(h_{\scriptscriptstyle{(17)}}\la b')a'\\
     =&a(h\o \la b)(h_{\scriptscriptstyle{(2)}}\la(g\o \la c))(h_{\scriptscriptstyle{(3)}}\la(g_{\scriptscriptstyle{(2)}}\la(f\o \la d)))(h_{\scriptscriptstyle{(4)}}\la(g_{\scriptscriptstyle{(3)}}\la\gamma(f_{\scriptscriptstyle{(2)}}, l\o )))(h_{\scriptscriptstyle{(5)}}\la\gamma(g_{\scriptscriptstyle{(4)}}, f_{\scriptscriptstyle{(3)}}l_{\scriptscriptstyle{(2)}}))\\
    &\quad(h_{\scriptscriptstyle{(6)}}\la\phi(g_{\scriptscriptstyle{(5)}}, f_{\scriptscriptstyle{(4)}}, l_{\scriptscriptstyle{(3)}}))\gamma(h_{\scriptscriptstyle{(7)}}, (g_{\scriptscriptstyle{(6)}}f_{\scriptscriptstyle{(5)}})l_{\scriptscriptstyle{(4)}})\phi(h_{\scriptscriptstyle{(8)}}, g_{\scriptscriptstyle{(7)}}f_{\scriptscriptstyle{(6)}}, l_{\scriptscriptstyle{(5)}})\phi(h_{\scriptscriptstyle{(9)}}, g_{\scriptscriptstyle{(8)}}, f_{\scriptscriptstyle{(7)}})\\
    &\quad\gamma^{-1}((h_{\scriptscriptstyle{(10)}}g_{\scriptscriptstyle{(9)}})f_{\scriptscriptstyle{(8)}}, l_{\scriptscriptstyle{(6)}})\gamma^{-1}(h_{\scriptscriptstyle{(11)}}g_{\scriptscriptstyle{(10)}}, f_{\scriptscriptstyle{(9)}})\gamma^{-1}(h_{\scriptscriptstyle{(12)}}, g_{\scriptscriptstyle{(11)}})(h_{\scriptscriptstyle{(13)}}\la(g_{\scriptscriptstyle{(12)}}\la(f_{\scriptscriptstyle{(10)}}\la d')))\\
    &\quad(h_{\scriptscriptstyle{(14)}}\la(g_{\scriptscriptstyle{(13)}}\la c'))(h_{\scriptscriptstyle{(15)}}\la b')a'\\
    =&a(h\o \la b)(h_{\scriptscriptstyle{(2)}}\la(g\o \la c))(h_{\scriptscriptstyle{(3)}}\la(g_{\scriptscriptstyle{(2)}}\la(f\o \la d)))(h_{\scriptscriptstyle{(4)}}\la(g_{\scriptscriptstyle{(3)}}\la \gamma(f_{\scriptscriptstyle{(2)}}, l\o )))(h_{\scriptscriptstyle{(5)}}\la\gamma(g_{\scriptscriptstyle{(4)}}, f_{\scriptscriptstyle{(3)}}l_{\scriptscriptstyle{(2)}}))\\
     &\quad\gamma(h_{\scriptscriptstyle{(6)}}, g_{\scriptscriptstyle{(5)}}(f_{\scriptscriptstyle{(4)}}l_{\scriptscriptstyle{(3)}}))\phi(g_{\scriptscriptstyle{(6)}}, f_{\scriptscriptstyle{(5)}}, l_{\scriptscriptstyle{(4)}})\phi(h_{\scriptscriptstyle{(7)}}, g_{\scriptscriptstyle{(7)}}f_{\scriptscriptstyle{(6)}}, l_{\scriptscriptstyle{(5)}})\phi(h_{\scriptscriptstyle{(8)}}, g_{\scriptscriptstyle{(8)}}, f_{\scriptscriptstyle{(7)}})\gamma^{-1}((h_{\scriptscriptstyle{(9)}}g_{\scriptscriptstyle{(9)}})f_{\scriptscriptstyle{(8)}}, l_{\scriptscriptstyle{(6)}})\\
     &\quad \gamma^{-1}(h_{\scriptscriptstyle{(10)}}g_{\scriptscriptstyle{(10)}}, f_{\scriptscriptstyle{(9)}})\gamma^{-1}(h_{\scriptscriptstyle{(11)}}, g_{\scriptscriptstyle{(11)}})(h_{\scriptscriptstyle{(12)}}\la(g_{\scriptscriptstyle{(12)}}\la(f_{\scriptscriptstyle{(10)}}\la d')))(h_{\scriptscriptstyle{(13)}}\la(g_{\scriptscriptstyle{(13)}}\la c'))(h_{\scriptscriptstyle{(14)}}\la b')a'.
\end{align*}
These are equal by axiom (2) of a coquasi-bialgebra $H$. \end{proof}

\begin{exa}\rm An example of Theorem~\ref{thm. coquasi CM bialgebroid} can be obtained from the data of a finite group $X$ with subgroup $G$ and a choice of transversal $M\subseteq X$ such that $X=GM$ is a unique factorisation. Then
\[ st=\tau(s,t) s\cdot t,\quad su=(s\la u)(s\ra u),\quad \forall\ s,t\in M,\ u\in G\]
defines a product $\cdot $ on $M$ making it into a quasigroup, a cocycle $\tau: M\times M\to G$, an action $\ra$ of $G$ on the set of $M$ and a cocycle-action $\la$ of $M$ on the set of $G$, in each case not by automorphisms but by `matched pair' conditions. We use the conventions of \cite{KM} where the conditions are written out in detail. Moreover, it is shown in this work that
there is a bicrossproduct coquasi-Hopf algebra $kM\rlbicross k(G)$
\[ (s\tens\delta_u)(t\tens\delta_v)=\delta_{u, t\la v} s\cdot t \tens\delta_v,\quad \Delta(s\tens \delta_u)=\sum_{ab=u}s\tens\delta_a \tens s\ra a\tens \delta_b\]
\[\eps(s\tens\delta_u)=\delta_{u,e},\quad \phi(s\tens\delta_u\tens t\tens\delta_v\tens r\tens\delta_w)=\delta_{u,\tau^{-1}(t,r)}\delta_{v,e}\delta_{w,e}\]
where $a,b\in G$ and $e$ is the relevant group identity. Our new results is that this acts canonically on $B=k G$, the group algebra of $G$, with
\[ (s\tens\delta_u)\la v=s\la v\delta_{u,v}.\]
One checks that this is a measuring, since
\[ (s\tens \delta_u)\la (vw)=\sum_{ab=u}((s\tens\delta_a)\la v) ((s\ra a\tens \delta_b)\la w)\]
as required. Moreover, we set
\[ \gamma(s\tens \delta_u\tens t\tens\delta_v)=\tau(s,t)\delta_{u,e}\tau(s,t),\quad \gamma^{-1}(s\tens \delta_u\tens t\tens\delta_v)=\tau(s,t)^{-1}\delta_{u,e}\tau(s,t),\]
these maps are convolution inverses of one another. Then with $h=s\tens\delta_u$,
\[h\o\tens h\t\tens h\th=(\Delta\tens\id)\Delta(s\tens\delta_u)=\sum_{abc=u} s\tens\delta_a\tens s\ra a\tens \delta_b\tens s\ra(ab)\tens\delta_c\]
for $a,b,c\in G$, and similarly for $g=t\tens\delta_v$, we have
\begin{align*}\sum_{abc=u, \bar a\bar b\bar c=v}&\gamma(s\tens\delta_a,t\tens\delta_{\bar a})(((s\ra a\tens\delta_b)(t\ra\bar a\tens\delta_{\bar b}))\la w)\gamma^{-1}(s\ra(ab)\tens \delta_c,t\ra(\bar a\bar b)\tens \delta_{\bar c})\\
&= \tau(s,t)((s\tens\delta_u)(t\tens\delta_v)\la w)\tau^{-1}(s\ra u, t\ra v)\\
&=\delta_{u,t\la v}\delta_{v,w}\tau(s,t) ( (s\cdot t)\la w)  \tau^{-1}(s\ra u, t\ra v)\\
&=\delta_{u,t\la v}\delta_{v,w}\tau(s,t) ( (s\cdot t)\la w)  \tau^{-1}(s\ra (t\la w), t\ra w)\\
&=\delta_{u,t\la v}\delta_{v,w} s\la(t\la w)=(s\tens\delta_u)\la ((t\tens \delta_v)\la w)
\end{align*}
as required, where $a=c=\bar a=\bar c=e$ due to the form of $\gamma,\gamma^{-1}$ so that $b=u$ and $\bar b=v$.
\end{exa}

We also record the following special case of Theorem~\ref{thm. coquasi CM bialgebroid}, where \(B\) is a strict left \(H\)-module algebra.

\begin{exa}\rm If \(B\) is a left \(H\)-module algebra for a coquasi-Hopf algebra \((H,\varphi)\), then Theorem~\ref{thm. coquasi CM bialgebroid} implies $(B\ot H\ot B, \Phi)$ is a left $B$-coquasi bialgebroid, with the (nonassociative unital) algebra structure:
    \[(a\ot h\ot a')(b\ot g\ot b')=a(h\o \la b)\ot h_{\scriptscriptstyle{(2)}}g\o \ot (h_{\scriptscriptstyle{(3)}}\la b')a',\]
    for any $h, g\in H$ and $a, a', b, b'\in B$.
    And the `invertible normalised 3-cocycle' $\Phi$
    \begin{align*}
        \Phi(X&, Y, Z)=a(h\o \la b)(h_{\scriptscriptstyle{(2)}}\la(g\o \la c))\phi(h_{\scriptscriptstyle{(3)}}, g_{\scriptscriptstyle{(2)}}, f\o )(h_{\scriptscriptstyle{(4)}}\la(g_{\scriptscriptstyle{(3)}}\la c'))(h_{\scriptscriptstyle{(5)}}\la b')a',
    \end{align*}
  for $X=a\ot h\ot a'$, $Y=b\ot g\ot b'$ and $Z=c\ot f\ot c'$.
\end{exa}
This also gives the following corollary.

\begin{cor}
Let \(B\) be a left \(H\)-module algebra for a coquasi-Hopf algebra \((H,\phi)\). Suppose that \(\gamma:H\otimes H\to B\) is a unital convolution-invertible map of associative type, 
in the sense that
 \[\gamma(\one{h}, \one{g})((\two{h}\two{g})\triangleright b)=((\one{h}\one{g})\triangleright b) \gamma(\two{h}, \two{g}),
\]
    for any $h, g\in H$ and $b\in B$. Then  $(B\ot H\ot B, \Phi^{\gamma})=(B\tens H\tens B,\Phi)^\Gamma$, where $\Gamma\in C^{2}(\cL, B)$ is given by
    \[
    \Gamma(X, Y)=a(h\o \la b)\gamma(h_{\scriptscriptstyle{(2)}}, g )(h_{\scriptscriptstyle{(3)}}\la b')a',
    \]
    for any $X=a\ot h\ot a', Y=b\ot g\ot b'\in B\ot H\ot B$. In particular, if $B$ is a left module algebra of $H$, and $\gamma$ is a 2-cocycle of associative type, then $(B\ot H\ot B)^{\Gamma}$ is a left and anti-left Hopf algebroid.
\end{cor}

\begin{proof}
    By Theorem \ref{thm. coquasi CM bialgebroid}, we know both $(B\ot H\ot B, \Phi)$ and $(B\ot H\ot B, \Phi^{\gamma})$ are left $B$-coquasi bialgebroids. We first check $\Gamma\in C^{2}(\cL, B)$. Since $\gamma(\one{h}, \one{g})((\two{h}\two{g})\triangleright b)=((\one{h}\one{g})\triangleright b) \gamma(\two{h}, \two{g})$ we can see $\Gamma(X, Yb)=\Gamma(X, Y\overline{b})$ for any $d\in B$. Its inverse is
    \[\Gamma^{-1}(X, Y)=a(h\o \la b)\gamma^{-1}(h_{\scriptscriptstyle{(2)}}, g\o )(h_{\scriptscriptstyle{(3)}}\la b')a'.\]
 Finally, by Theorem \ref{twistprod1}, it is not hard to see
    \begin{align*}       X\cdot_{\Gamma}Y=a(h\o \la b)\gamma(h_{\scriptscriptstyle{(2)}}, g\o )\ot h_{\scriptscriptstyle{(3)}}g_{\scriptscriptstyle{(2)}}\ot \gamma^{-1}(h_{\scriptscriptstyle{(4)}}, g_{\scriptscriptstyle{(3)}})(h_{\scriptscriptstyle{(5)}}\la b')a',
    \end{align*}
    and
    \begin{align*}
        \Phi^{\Gamma}(X&, Y, Z)=a(h\o \la b)(h_{\scriptscriptstyle{(2)}}\la(g\o \la c))(h_{\scriptscriptstyle{(3)}}\la\gamma(g_{\scriptscriptstyle{(2)}}, f\o ))\gamma(h_{\scriptscriptstyle{(4)}}, g_{\scriptscriptstyle{(3)}}f_{\scriptscriptstyle{(2)}})\\
    &\phi(h_{\scriptscriptstyle{(5)}}, g_{\scriptscriptstyle{(4)}}, f_{\scriptscriptstyle{(3)}})\gamma^{-1}(h_{\scriptscriptstyle{(6)}}g_{\scriptscriptstyle{(5)}}, f_{\scriptscriptstyle{(4)}})\gamma^{-1}(h_{\scriptscriptstyle{(7)}}, g_{\scriptscriptstyle{(6)}})(h_{\scriptscriptstyle{(8)}}\la(g_{\scriptscriptstyle{(7)}}\la c'))(h_{\scriptscriptstyle{(9)}}\la b')a',
    \end{align*}
  for $X=a\ot h\ot a'$, $Y=b\ot g\ot b'$ and $Z=c\ot f\ot c'$. The remaining assertion follows from \cite[Theorem 3.8 and Remark 3.9]{HM22} since $B\ot H\ot B$ is a left and anti-left Hopf algebroid.
\end{proof}

\section{\bf Dual constructions }\label{secdual}

We conclude by explaining the dual picture. The constructions so far are coquasi in nature: the coproduct remains coassociative while the product is associative only 
up to a \(3\)-cocycle. Under suitable finite projectivity assumptions, one may dualise this picture. The result is a quasi-bialgebroid, where the algebra remains associative 
but the coproduct becomes coassociative only up to an invertible reassociator.

\subsection{Quasi-bialgebroids}

We now formulate the dual constructions corresponding to those of Section~\ref{seccoquasi}.

\begin{defi}\label{def:coright.bgd1} Let $B$ be a unital associative algebra. A left $B$-quasi-bialgebroid is a unital associative algebra $\cL$, commuting (`source' and `target') algebra maps $s:B\to \cL$ and $t:B^{op}\to \cL$ (and hence a $B^e$-bimodule) and a $B$-bimodule structure (\ref{eq:rbgd.bimod}) with a pair of bimodule maps (called coproduct and counit)
\[\Delta:\cL\to \cL\di\cL,\quad \varepsilon:\cL\to B\]
satisfy the following compatibility conditions:
\begin{itemize}
\item[(i)] The coproduct $\Delta$ corestricts to an algebra map  $\cL\to \cL\times_B \cL$ where
\begin{equation*} \cL\times_{B} \cL :=\{\ \sum_i X_i \ot_{B} Y_i\ |\ \sum_i X_i\,\overline{b} \ot_{B} Y_i=
\sum_i X_i \ot_{B}  Y_i \,b,\quad \forall b\in B\ \}\subseteq \cL\tens_B\cL,
\end{equation*}
 is an algebra via factorwise multiplication.
\item[(ii)] The counit $\varepsilon$ is a `left character' in the sense
\begin{equation*}\varepsilon(1_{\cL})=1_{B}, \quad \varepsilon(X\varepsilon(Y))=\varepsilon(XY)=\varepsilon(X\overline{\varepsilon(Y)})\end{equation*}
for all $X,Y\in \cL$ and $b\in B$.
\item[(iii)] $\overline{\varepsilon(X\t)}X\o=X=\varepsilon(X\o)X\t$, for all $X\in\cL$.
\item[(iv)] There is an invertible counital \(3\)-cocycle on \(\cL\), namely an invertible element $\Phi=\Phi^{1}\ot\Phi^{2}\ot\Phi^{3}\in \cL\times_{B}\cL\times_{B}\cL$ such that
\begin{align*}(1)\quad &(\varepsilon\ot_{B}\id\ot_{B}\id)\Phi=(\id\ot_{B}\varepsilon\ot_{B}\id)\Phi=(\id\ot_{B}\id\ot_{B}\varepsilon)\Phi=1,\\
(2)\quad & s(b)\Phi^{1}\ot_{B}\Phi^{2}\ot_{B}t(b')\Phi^{3}=\Phi^{1}s(b)\ot_{B}\Phi^{2}\ot_{B}\Phi^{3}t(b'), \forall b, b'\in B,\\
(3)\quad & (\id\ot_{B}\Delta)(\Delta(X))\Phi=\Phi(\Delta\ot_{B}\id)(\Delta(X)),\\
(4)\quad & (\id\ot_{B}\id_{B}\ot_{B}\Delta)(\Phi)(\Delta\ot_{B}\id_{B}\ot_{B}\id)(\Phi)=(1\ot_{B}\Phi)(\id\ot_{B}\Delta\ot_{B}\id)(\Phi)(\Phi\ot_{B} 1)
\end{align*}
for all $X\in \cL$.
\end{itemize}
\end{defi}

\begin{thm}\label{twistprod2}
    Let $(\cL, \Phi)$ be a left $B$-quasi-bialgebroid, and $F$ be a  counital invertible element in $\cL\times_{B}\cL$, such that $s(b)F^{\alpha}\ot t(b')F_{\alpha}=F^{\alpha}s(b)\ot F_{\alpha}t(b')$. If we define a new coproduct on $\cL$ by
    \begin{align*}
    \Delta^{F}(X):=F\Delta(X)F^{-1},
\end{align*}
then the underlying algebra $\cL$ with the new coproduct is a left $B$-quasi-bialgebroid $\CL_{F}$ with the twisted 3-cocycle given by
\begin{align*}
   (1\di F)((\id\di\Delta)F)\Phi((\Delta\di\id)\Phi^{-1})(\Phi\di 1),
\end{align*}
for all $X\in \cL$.
\end{thm}
\proof The verification is straightforward and is omitted, since it follows the same pattern as the usual construction over a field; see, for example, \cite{Ma:book}. \endproof
Recall that given a left $B^e$-module $M$, $M^{\vee}:=\Hom_{\BB-}(M,\,B)$. By using \cite[Theorem 5.13]{schau1},  we  have the following lemma:
\begin{lem}
    Let $\cL$ be a left coquasi-bialgebroid over $B$ which is a finitely generated projective as a right B module induced by the target map on the left, then $\Lambda:=\cL^{\vee}$ is a left quasi bialgebroid over $B$.  Moreover, we have $(\cL^{\Gamma_{F}^{-1}})^{\vee}\simeq \Lambda_{F}$ for any $F\in \Lambda\times_{B}\Lambda$ as in Theorem \ref{twistprod2}.
\end{lem}
\begin{proof}
    The 3-cocycle correspondence can be given by the following isomorphism: $\hat{\Phi}: \cL^{\vee}\times_{B}\cL^{\vee}\times_{B}\cL^{\vee}\to \Hom_{B^{e}-}(\cL\otimes_{B^e}\cL\otimes_{B^e}\cL, B)$ by
    \[\hat{\Phi}(\alpha\ot_{B}\beta\ot_{B}\gamma)(X\ot_{B^{e}}Y\ot_{B^{e}}Z)=\lan\alpha|X t\lan \beta|Y t\lan \gamma|Z\ran\ran\ran.\]
    Moreover, given $F=F^{\alpha}\ot F_{\alpha}\in \Lambda\times_{B}\Lambda$, we can construct $\Gamma_{F}\in C^2(\cL, B)$ by 
    \[\Gamma_{F}(X, Y)=\<F^{\alpha}|X\overline{\<F_{\alpha}|Y\>}\>.\]
We will omit the details for the rest of the proof.
\end{proof}

\subsection*{Acknowledgements}  XH was supported by Marie Curie Fellowship HADG - 101027463 agreed between QMUL  and the  European Commission.

\medskip\noindent{\bf Data availability}  Data sharing is not applicable as no data sets were generated or analysed during the current
 study.

\section*{Declarations}

{\bf Conflict of interest} On behalf of all authors, the corresponding author states that there is no conflict of interests.

\appendix

\renewcommand\refname

\end{document}